\documentclass{amsart}
\usepackage[utf8]{inputenc}

\usepackage{mathtools}
\usepackage{amssymb}
\usepackage{accents}
\usepackage{mathrsfs}
\usepackage{bm}
\usepackage[all]{xy}
\UseComputerModernTips
\CompileMatrices
\usepackage[svgnames]{xcolor}
\usepackage[colorlinks,allcolors=blue]{hyperref}
\usepackage{ifthen}
\usepackage{latexsym}
\usepackage{enumitem}
\setlist[enumerate,1]{label=\textup{(\arabic*)}}
\usepackage{tikz}
\usepackage{tikz-3dplot}
\usetikzlibrary{backgrounds}
\usetikzlibrary{arrows.meta}
\allowdisplaybreaks
\numberwithin{equation}{section}
\newtheorem{theorem}[equation]{Theorem}
\newtheorem{lemma}[equation]{Lemma}
\newtheorem{proposition}[equation]{Proposition}
\newtheorem{corollary}[equation]{Corollary}

\theoremstyle{definition}
\newtheorem{definition}[equation]{Definition}
\newtheorem{example}[equation]{Example}

\theoremstyle{remark}
\newtheorem{remark}[equation]{Remark}

\renewcommand{\phi}{\varphi}
\newcommand{\bndry}{\partial}
\DeclareMathSymbol{\boxprod}{\mathbin}{AMSa}{"03} 
\DeclareMathSymbol{\mixprod}{\mathbin}{AMSa}{"4F} 

\newcommand{\dirsum}{\oplus}

\newcommand{\disjunion}{\sqcup}

\newcommand{\dual}{^\vee}

\newcommand{\includesin}{\hookrightarrow}
\newcommand{\intersect}{\cap}
\newcommand{\iso}{\cong}
\newcommand{\Mackey}[1]{{\underline {#1}}}

\newcommand{\onto}{\twoheadrightarrow}

\newcommand{\susp}{\Sigma}
\newcommand{\tensor}{\otimes}
\newcommand{\union}{\cup}

\newcommand{\C}{{\mathbb C}}

\newcommand{\PP}{\mathbb{P}}
\newcommand{\R}{{\mathbb R}}

\newcommand{\Z}{\mathbb{Z}}
\newcommand{\ZZ}{\mathbb{Z}}
\newcommand{\HS}{\mathbb{H}}
\newcommand{\HH}{H_\GG}

\newcommand{\cwd}[1][]{\widehat{c}_\omega^{\ifthenelse{\equal{#1}{}}{}{{\:#1}}}}
\newcommand{\cxwd}[1][]{\widehat{c}_{\chiw}^{\ifthenelse{\equal{#1}{}}{}{{\:#1}}}}

\newcommand{\cld}[1][]{\widehat{c}_{\lambda}^{\ifthenelse{\equal{#1}{}}{}{{\:#1}}}}
\newcommand{\cxld}[1][]{\widehat{c}_{\chi\lambda}^{\ifthenelse{\equal{#1}{}}{}{{\:#1}}}}

\newcommand{\clod}[1][]{\widehat{c}_{\omega_1}^{\ifthenelse{\equal{#1}{}}{}{{\:#1}}}}
\newcommand{\cxlod}[1][]{\widehat{c}_{\chi\omega_1}^{\ifthenelse{\equal{#1}{}}{}{{\:#1}}}}
\newcommand{\cltd}[1][]{\widehat{c}_{\omega_2}^{\ifthenelse{\equal{#1}{}}{}{{\:#1}}}}
\newcommand{\cxltd}[1][]{\widehat{c}_{\chi\omega_2}^{\ifthenelse{\equal{#1}{}}{}{{\:#1}}}}
\newcommand{\cltensd}[1][]{\widehat{c}_{\omega_1\tensor\omega_2}^{\ifthenelse{\equal{#1}{}}{}{{\:#1}}}}
\newcommand{\cxltensd}[1][]{\widehat{c}_{\chi\omega_1\tensor\omega_2}^{\ifthenelse{\equal{#1}{}}{}{{\:#1}}}}
\newcommand{\cd}[1][]{\widehat{c}^{\ifthenelse{\equal{#1}{}}{}{{\:#1}}}}
\newcommand{\cgd}[1][]{\widehat{c}_\pi^{\ifthenelse{\equal{#1}{}}{}{{\:#1}}}}
\newcommand{\cxgd}[1][]{\widehat{c}_{\chi\pi}^{\ifthenelse{\equal{#1}{}}{}{{\:#1}}}}
\newcommand{\Cpq}[2]{\C^{#1+#2\sigma}}
\newcommand{\Cp}[1]{\C^{#1}}
\newcommand{\Cq}[1]{\C^{#1\sigma}}
\newcommand{\Xpq}[2]{\PP(\Cpq{#1}{#2})}

\newcommand{\Xp}[1]{\PP(\C^{#1})}
\newcommand{\Xq}[1]{\PP(\Cq{#1})}
\newcommand{\chiw}{\chi\omega}
\newcommand{\Grpq}[3]{\mathrm{Gr}_{#1}(\Cpq{#2}{#3})}
\newcommand{\Grp}[2]{\mathrm{Gr}_{#1}(\Cp{#2})}

\newcommand{\xQp}[1]{\chi Q_{#1}}
\newcommand{\Qexp}[1]{Q(\Cp {#1})}

\newcommand{\xQexp}[1]{\chi Q(\Cpq {#1}{#1})}
\renewcommand{\vec}[1]{\accentset{\rightharpoonup}{#1}} 
\newcommand{\gr}{\Diamond}      
\newcommand{\ext}{\mathsf{\Lambda}}     
\newcommand{\rels}[1]{\langle #1 \rangle}

\DeclareMathOperator{\grad}{grad}
\DeclareMathOperator{\Sym}{Sym}

\DeclareMathOperator{\GL}{GL}
\newcommand{\GG}{{C_2}}
\begin{document}

\title[Complex quadrics I]
{The $\GG$-equivariant ordinary cohomology of complex quadrics I: The antisymmetric case}


\author{Steven R. Costenoble}
\address{Steven R. Costenoble\\Department of Mathematics\\Hofstra University\\
  Hempstead, NY 11549, USA}
\email{Steven.R.Costenoble@Hofstra.edu}
\author{Thomas Hudson}
\address{Thomas Hudson, College of Transdisciplinary Studies, DGIST, 
Daegu, 42988, Republic of Korea}
\email{hudson@dgist.ac.kr}



\abstract
In this, the first of three papers about $\GG$-equivariant complex quadrics,
we calculate the equivariant ordinary cohomology of smooth antisymmetric quadrics.
One of these quadrics coincides with a $\GG$-equivariant Grassmannian,
and we use this calculation to prove an equivariant refinement of the result that there are 27 lines
on a cubic surface in $\PP^3$.
\endabstract

\maketitle
\tableofcontents

\section{Introduction}\label{sec:introduction}
This paper and its two sequels \cite{CH:QuadricsII,CH:QuadricsIII} are devoted to the study of the additive and multiplicative structures of the $\GG$-equivariant ordinary cohomology of non-degenerate complex quadrics endowed with a linear $\GG$-action. In the non-equivariant case, up to change of coordinates, there are essentially two cases, depending on the parity of the dimension of the ambient projective space in which the quadric lives. Inside $\Xp{2p}:=\{[x_1:\cdots:x_p:y_p:\cdots:y_1] \mid x_j,y_j\in\C\}$ one considers the quadric $Q_{2p}$ defined as the zero locus of the polynomial
\begin{equation}\label{evenquadric}
    f=x_1y_1+x_2y_2+\cdots+x_py_p,
\end{equation}
while in $\Xp{2p+1}=\{[x_1:\cdots:x_p:z:y_p:\cdots:y_1] \mid x_j,z,y_j,\in\C\}$ the quadric $Q_{2p+1}$ is the zero locus of  
\begin{equation}\label{oddquadric}
    f=x_1y_1+x_2y_2+\cdots+x_py_p+z^2.
\end{equation}

In both cases the resulting variety can be given a cellular decomposition whose cells correspond to an additive basis for the cohomology. The decomposition depends on a choice of a maximal flag
$\{F_0\subset F_1\subset \cdots \subset F_p\}$
of isotropic subspaces, \textit{i.e.,} subspaces on which the relevant polynomial vanishes identically. The standard choice is to consider 
\[
    F_i:=\bigl\{(x_1,\dots,x_i,0,\dots,0) \mid x_j\in \C\bigr\}
\]
for $i\in\{0,1,\dots,p\}$ and then complete the flag by setting $F_{p+i}:=F_{p-i}^\perp$.
Here, for a subspace $V$,  we denote by $V^\perp$ the subspace  of vectors  which are orthogonal to the elements of $V$ with respect to the symmetric form associated to $f$. 
When one considers the multiplicative structure, an interesting feature appears: Unlike for projective space, the cohomology ring of the quadric is not generated solely by $\cd:=c_1(O(1))$, the first Chern class of the dual of the tautological line bundle.
It is necessary to add as additional generators the classes dual to the projectivizations of some maximal isotropic  subspaces. In the odd case $Q_{2p+1}$ this issue can be bypassed by inverting 2, but in the even one it is unavoidable.

Being the case most relevant to this paper, let us describe in detail the cohomology ring of $Q_{2p}$, as calculated for instance in \cite{EdidinGraham:quadricbundles}. In order to illustrate the parallel with the equivariant case, we introduce the following notation. For $s\in\{0,1,\dots,p\}$ we consider the maximal isotropic subspace
\[
    M_s:=\bigl\{(x_1,\dots,x_{p-s},0,\dots,0,y_p,\dots,y_{p-s+1},0,\dots,0) \mid x_j,y_j\in \C\bigr\} \subset \Cp{2p}\,,
\]
where the index $s$ represents the codimension of $M_s\intersect F_p$ in $F_p$,
the maximal isotropic subspace in the reference flag.
Edidin and Graham in \cite{EdidinGraham:quadricbundles} describe the cohomology as an algebra over $\Z = H^*(*;\Z)$,
but here we describe it as a cohomology over $R^*:=H^*(BU(1),\Z)\iso\Z[\cd]$, which will be more convenient
for the comparison with the equivariant case.
As an algebra over  $R^*$,
the cohomology of  $H^*(Q_{2p},\Z)$ is given by the quotient 
\begin{equation}
R^*\bigl[m_{[0]},m_{[1]}\bigr]/I,
\end{equation}
where $m_{[0]}$ and $m_{[1]}$ are, respectively, the cohomology classes $[\PP(M_0)]^*$ and $[\PP(M_1)]^*$ 
Poincar\'e dual to $\PP(M_0)$ and $\PP(M_1)$ (which we shall refer to in short as their fundamental classes),
and $I$ is the ideal generated by the relations 
\begin{equation}\label{relnoneq}
    \text{(i)}\ \cd[p-1]=m_{[0]}+m_{[1]}\,, \qquad \text{(ii)}\ \ \cd\,m_{[0]}=\cd\,m_{[1]}
\end{equation}
\begin{align*}
    \mathllap{\text{and}\qquad}
    \text{(iii)} \ m_{[0]}\,m_{[1]}=
    \begin{cases} 
        \hfill 0 \hfill & \text{if $p$ is odd}\\
        \cd[p-1]m_{[0]}&\text{if $p$ is even.}
    \end{cases}
\end{align*}
The notation $m_{[s]}$ indicates that the fundamental class $[\PP(M_s)]^*$ depends only on the parity $[s]$ of
the codimension $s$,
so that we could take, for example, $m_{[0]}$ to be represented by 
any $\PP(M_s)$ with $s$ even.

The goal of this paper is to generalise this presentation to the case in which $Q_{2p}$ is made into a $\GG$-space
(denoted $\xQp{2p}$) by considering it as a subspace of $\Xpq{p}{p}$, the projectivisation of the direct sum of $p$ copies of $\C$ with the trivial action of $\GG$ and $p$ copies of $\C^\sigma$, the complex numbers with the sign action.
With our previous notations, the $x_j$s will be the variables with trivial action, while $\GG$ acts as $-1$ on the $y_j$s. We refer to the resulting quadric $\xQp{2p}$ as \textit{antisymmetric}, because the action of $\GG$ on the polynomial ring maps $f$ to $-f$. It is not difficult to see that, if a nondegenerate antisymmetric quadric sits inside $\Xpq pq$, then we must have $p=q$, because variables with different actions need to be paired with one another.

Some changes in the description of the cohomology are dictated by the different setting. The cohomology theory we use is the equivariant ordinary cohomology with extended grading developed in \cite{CostenobleWanerBook}; see the Appendix for a brief overview. Thus, we will describe the cohomology as an algebra over $\mathcal{R}^\diamond:=\HH^\gr(BU(1))$, the ordinary cohomology of the infinite projective space $\Xpq{\infty}{\infty}$ with its extended grading~$\diamond$, as computed in \cite{Co:InfinitePublished} and summarized in the Appendix. 
This ring is generated by classes $\cwd$ and $\cxwd$, the Euler classes of 
$\omega\dual$ and $\omega\dual\tensor\Cq{}$,
respectively, where $\omega$ is the tautological line bundle over $\Xpq{\infty}{\infty}$,
and two more classes $\zeta_0$ and $\zeta_1$.

But there are more interesting changes as well.
As we have seen,
in the classical setting, the maximal isotropic subspaces $M_s$ divide into two classes, depending on the parity of their intersections with the reference space $F_p$, with members of one class representing $m_{[0]}$ and members of the other representing $m_{[1]}$.
In the equivariant setting this dichotomy breaks down further:
The projectivizations of the subspaces $M_s$ have different $\GG$-actions and give rise to distinct cohomology classes. For $s\in\{0,\dots,p\}$ we will denote by $m_s$ the equivariant fundamental class $[\PP(M_s)]^*$. We are now in position to state our main theorem. 

\begin{theorem}\label{thm:intro}
As an algebra over $\mathcal{R}^\diamond$, the ordinary cohomology $\HH^\gr(\xQp{2p})$ is generated
by the fundamental classes
\[
m_s:=[\PP(M_s)]^*
\]
with $s\in\{0,\dots,p\}$. These classes are such that
$\cwd[p-s] m_s$ is infinitely divisible by $\zeta_0$, while $\cxwd[s] m_s$
is infinitely divisible by $\zeta_1$. Furthermore, the following  relations hold: 
\begin{alignat*}{3}
 \text{(i)}&\qquad& \cwd[s]\,\cxwd[p-1-s] &= \zeta_0 m_{s+1} + \zeta_1 m_s &\qquad&\text{for $0\leq s \leq p-1$}\,; \\
 \text{(ii)} &&  \cxwd m_{s+1} &= \cwd m_{s} &&\text{for $0 \leq s \leq p-1$}\,; \\
 \text{(iii)} && m_{s} m_{p-s} &= 0 &&\text{for $0\leq s \leq p$\,. } 
\end{alignat*}
These relations generate all the relations in the algebra. \qed
\end{theorem}

The class $m_s$ restricts nonequivariantly to $m_{[s]}$.
Because $\cwd$ and $\cxwd$ are both mapped to $\cd$ and $\zeta_0$ and $\zeta_1$ both restrict to 1,
the first two relations  directly recover their nonequivariant counterparts (\ref{relnoneq},\,i,\,ii). The third equality requires some more interpretation since, depending on the parities of $p$ and $s$, it will restrict to the vanishing of either $m_{[0]}^2$, $m_{[0]}\,m_{[1]}$, or $m_{[1]}^2$. Each of these expressions has a formula analogous to (\ref{relnoneq},\,iii) that gives 0 for the appropriate choice of parities.
This is discussed in more detail in Section \ref{sec:restrictions}.

Since  the resolution of the Kervaire invariant problem
by Hill, Hopkins, and Ravenel in \cite{HHRKervaire}, in which calculations in Bredon's ordinary equivariant cohomology played a prominent role,
there has been renewed interest in the subject.
Most published computations such as
\cite{Dug:AHSSforKR}, \cite{DuggerGrass}, \cite{HazelFundamental}, \cite{HazelSurfaces}
\cite{Hogle}, \cite{HoglePoincarepoly}, and \cite{KronholmSerre}, concentrate on the $RO(G)$-graded theory.
However, the $RO(G)$-graded viewpoint may be too narrow
and we have found that enlarging the grading leads to simpler, more natural presentations.
Previous work using this approach includes
\cite{Co:InfinitePublished}, \cite{CHTFiniteProjSpace},
\cite{CH:bt2}, and \cite{CH:bu2},
as well as \cite{CHTAlgebraic} and \cite{CH:geometric}, which use the extended grading
to obtain two different equivariant refinements of the classical  B\'ezout theorem.
As a concrete example in the present case, in dealing with quadrics when $p$ is odd, none of the generating elements $m_s$ lives in the
$RO(\GG)$ grading and, if $p$ is even, only one does, $m_{p/2}$.
This makes describing the structure of just the $RO(\GG)$-graded part much more
complicated than the statement of the theorem above.
We illustrate this in Section~\ref{sec:main theorem}, after proving Theorem~\ref{thm:intro}.

One remarkable feature of equivariant ordinary cohomology is that it captures in a single functor information about both the non-equivariant cohomology of the given space and the cohomology of its fixed sets. But this also makes it difficult to compute, with the level of the challenge increasing hand in hand with the topological complexity of the fixed sets. For example, an increase in the number of connected components leads to an enlarging of the grading one has to consider, even when those components are simply connected.

This partially explains our decision to consider quadrics before finite Grassmannians, perhaps a more obvious target after considering finite projective spaces in \cite{CHTFiniteProjSpace}, particularly from the point of view of algebraic topology. The set of fixed points of a Grassmannian $\Grpq dpq$ with $d\leq p,q$  has $d+1$ connected components, which translates into a grading group isomorphic to $\Z^{d+2}$. On the other hand, most quadrics have fixed sets with two components, as do $BU(1)$ and (most) finite projective spaces, and are therefore graded on the same $\Z^3$ used for projective spaces. The only exceptions are those quadrics that, for dimensional reasons, have some zero-dimensional fixed-set components; they can have fixed sets with up to four components and will be treated separately in \cite{CH:QuadricsIII}.

The difference in the geometries of the fixed sets is also the reason we have separated the antisymmetric and the symmetric cases, the latter of which will be treated in \cite{CH:QuadricsII}. 
 While in the antisymmetric case every  quadric contains the  fixed sets of its ambient projective space, namely two disjoint copies of $\Xp p$, symmetric quadrics intersect the fixed projective spaces in smaller quadrics. 
Since the parities of the dimensions of the fixed projective spaces in which they live vary, we can see the polynomials (\ref{evenquadric}) and (\ref{oddquadric}) both occur, leading to four different cases that need to be considered.

Another reason to be interested in quadrics is that,
from the perspective of algebraic geometry and more specifically Schubert calculus, they are natural objects to consider and usually appear quite high on the to-do list of computations to be performed when a new functor---in our case equivariant ordinary cohomology---is introduced. If, on the one hand, they are the simplest examples of algebraic hypersurfaces after projective spaces, they also play a key role in Schubert calculus. In fact, for every classical group it is possible to define analogues of the classical Grassmannians and, more generally, flag varieties by interpreting them as quotients of $\GL(n,\C)$ by parabolic subgroups. In the case of the orthogonal groups $O(n,\C)$, the corresponding quotients give quadrics and quadric bundles in place of projective spaces and projective bundles.

As a sample application of our calculations, we will use them to refine the famous classical result in enumerative geometry
that there are exactly 27 lines on a smooth cubic hypersurface $H$ in $\Xp 4$. The cohomological translation of this counting problem requires a computation in the cohomology of $\Grp 24$ and, thanks to the fact that this Grassmannian happens to be a quadric in $\Xp 6$, in Section~\ref{sec:27 lines} we are able to use our calculations to answer the question,
\begin{equation}\label{27 lines question}
    \text{``How does $\GG$ act on the 27 lines if $H$ is contained in $\Xpq{3}{}$?''}
\end{equation}
On the left in Figure~\ref{fig:lines}, we show the three ways in which $\GG$ can act on an affine plane,
and the induced actions on a projective line. 
In the top part of the figure, showing the real part of the  affine planes, blue indicates trivial $\GG$-action while red indicates action by $-1$;
the resulting projective lines are shown in the bottom part of the figure.
Thought of as points in a Grassmannian $\Grpq 2pq$,
these actions are distinguished by the fact that the corresponding points lie in different components of
the fixed set $\Grpq 2pq^\GG$, and our cohomology computation will be able to make this
distinction as well. Hence, for those lines in the cubic $H$ that are $\GG$-invariant, we will be able
to tell exactly how $\GG$ acts on them.
The other possibility is that some of the lines may not be $\GG$-invariant, but then $\GG$ must
take such a line to another line in $H$. These lines thus pair up and form
free $\GG$-orbits in the Grassmannian. Since the Grassmannian is nonequivariantly connected,
any two maps of the free orbit $\GG/e$ into the Grassmannian are $\GG$-homotopic, so cohomology
cannot make any distinction between them, and they all appear generically like the free action
shown on the right in Figure~\ref{fig:lines}.

Our result recovers a calculation done by Brazelton in \cite{Braz:equivenumerative} that described the action of $\GG$ on the set of 27 lines, but refines that calculation by distinguishing the types of lines that appear;
we refer the reader to Section~\ref{sec:27 lines} for an in depth comparison of our approach with Brazelton's.
We will consider the same question for a cubic in $\Xpq 22$ in \cite{CH:QuadricsIII}.

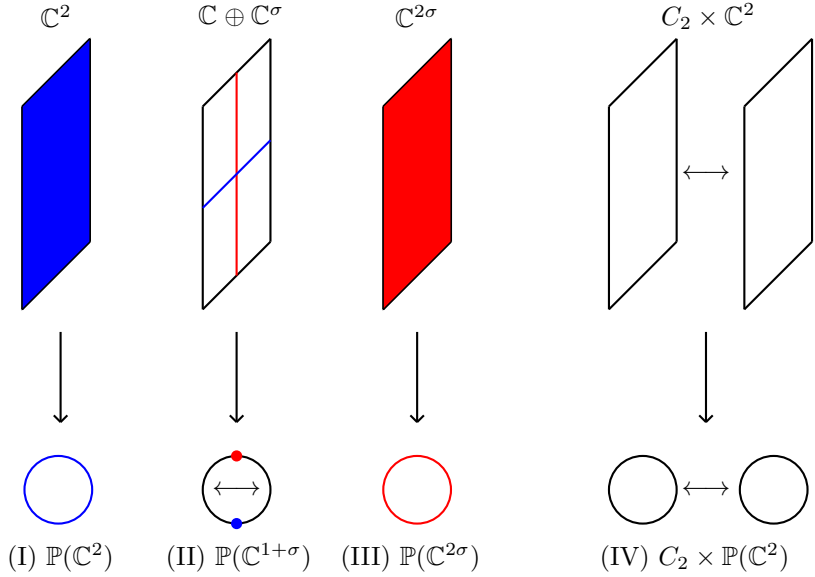
\begin{figure*}\label{fig:lines}
\phantom{a}\hspace{-2 cm}
\begin{tikzpicture}[scale=0.3]
\draw[thick] (-10,7) -- (-7,10);
\draw[thick] (-10,-2) -- (-7,1);
\draw[thick] (-10,7) -- (-10,-2);
\draw[thick] (-7,10) -- (-7,1); 
\draw[fill=blue] (-10,7) -- (-7,10) -- (-7,1) -- (-10,-2) --(-10,7);

\draw[thick] (-8.3,-3) -- (-8.3,-7);
\draw[thick] (-8.6,-6.7) -- (-8.3,-7);
\draw[thick] (-8,-6.7) -- (-8.3,-7);

\draw[thick, blue] (-8.4,-10) circle (1.5);

\draw[thick] (-2,7) -- (1,10);
\draw[thick] (-2,-2) -- (1,1);
\draw[thick] (-2,7) -- (-2,-2);
\draw[thick] (1,10) -- (1,1);
\draw[thick,red]  (-0.5,8.5) -- (-0.5,-0.5);
\draw[thick,blue] (-2,2.5) -- (1,5.5); 

\draw[thick] (-0.5,-3) -- (-0.5,-7);
\draw[thick] (-0.8,-6.7) -- (-0.5,-7);
\draw[thick] (-0.2,-6.7) -- (-0.5,-7);

 \draw[thick] (-0.5,-10) circle (1.5);
\draw[thick,blue,fill=blue] (-0.5,-11.5) circle (0.2);
\draw[thick,red, fill=red] (-0.5,-8.5) circle (0.2);

\draw[thick] (6,7) -- (9,10);
\draw[thick] (6,-2) -- (9,1);
\draw[thick] (6,7) -- (6,-2);
\draw[thick] (9,10) -- (9,1);
\draw[fill=red] (6,7) -- (9,10) -- (9,1) -- (6,-2) --(6,7); 

\draw[thick] (7.5,-3) -- (7.5,-7);
\draw[thick] (7.2,-6.7) -- (7.5,-7);
\draw[thick] (7.8,-6.7) -- (7.5,-7);

\draw[thick,red] (7.5,-10) circle (1.5);

\draw[thick] (16,7) -- (19,10);
\draw[thick] (16,-2) -- (19,1);
\draw[thick] (16,7) -- (16,-2);
\draw[thick] (19,10) -- (19,1);

\draw[thick] (22,7) -- (25,10);
\draw[thick] (22,-2) -- (25,1);
\draw[thick] (22,7) -- (22,-2);
\draw[thick] (25,10) -- (25,1);

\draw[thick] (20.3,-3) -- (20.3,-7);
\draw[thick] (20,-6.7) -- (20.3,-7);
\draw[thick] (20.6,-6.7) -- (20.3,-7);

\draw[thick, ] (17.5,-10) circle (1.5);

\draw[thick] (23.3,-10) circle (1.5);


 \node at (-8.5,11) {$ \Cp2$};
\node at (-8.3,-13) {(I)$\  \Xp2$};

\node at (-0.5,-10) {$\longleftrightarrow$};
\node at (-0.3,11) {$ \Cp{}\oplus\Cq{}$};
\node at (-0.4,-13) {(II)$\  \Xpq{1}{}$};
  
\node at (7.5,11) {$\Cq2$};
\node at (7.2,-13) {(III)$\ \Xq{2}$};
 
\node at (20.4,11) {$\GG\times\Cp2$}; 
\node at (20.3,4) {$\longleftrightarrow$};
\node at (20.3,-10) {$\longleftrightarrow$};
\node at (19.8,-13) {(IV)$\ \GG\times\Xp2$};    


\end{tikzpicture}
\caption{$\GG$-affine planes and their associated projective lines}\label{fig:lines}
\end{figure*}

The structure of the remainder of the paper is as follows.
We restate and prove Theorem~\ref{thm:intro} in 
Section~\ref{sec:main theorem}, as Theorem~\ref{thm:xo2 mutliplicative},
by first finding splittings that allow us to determine the additive structure,
and then finding sufficient relations to determine the multiplicative structure.
In Section~\ref{sec:restrictions}, we investigate two restrictions maps,
the restriction to nonequivariant cohomology and the fixed-point map.
Besides helping to explicate the connection between the equivariant result and the classical one,
these restriction maps are useful in calculations, including those that follow in subsequent sections.
One of the quadrics we study here, $\xQp{6}$, coincides with the Grassmannian $\Grpq 23{}$,
and we spell out the connection in more detail in Section~\ref{sec:grassmannian}.
In Section~\ref{sec:27 lines}, we use the resulting calculation of the cohomology of this Grassmannian
to answer question (\ref{27 lines question}) above about the 27 lines on a cubic in $\Xpq 3{}$.
We also include an Appendix in which we summarize results about the
equivariant ordinary cohomology theory we use, earlier results on the cohomology
of equivariant projective spaces, and the use of equivariant singular manifolds
to represent cohomology classes.

\subsection*{Acknowledgements}
Both authors would like to thank Sean Tilson for his help
in laying the foundations of this work. The first author thanks Hofstra University for
a research grant providing financial support for this project.
The second author was partially supported by the National Research Foundation of Korea (NRF) grant funded by the Korea government (MSIT) (No. RS-2024-00414849).

\section{The cohomology of antisymmetric quadrics}\label{sec:main theorem}

We now turn to our main task, the computation of the cohomology of $\xQp{2p} = \xQexp{p}$.
Via the inclusions $\xQexp{p}\includesin \Xpq pp \includesin BU(1)$,
we may consider $\HH^\gr(\xQp{2p})$ as an algebra over $\HH^\gr(BU(1))$,
and $\HH^\gr(\xQp{2p})$ contains restrictions of the generators $\zeta_0$, $\zeta_1$, $\cwd$, and $\cxwd$,
which we will call by those same names.

We mentioned earlier that the fixed sets are projective spaces:
\[
    \xQp{2p}^\GG = \Xp p \disjunion \Xq p.
\]
These come from two maximal $\GG$-invariant isotropic subspaces,
and are special cases of the following definition that highlights some other such subspaces.
In this definition we use the notation
\[
    \vec x = (x_1,x_2,\cdots,x_s)
\]
for a vector, and the notation $\vec 0_k$ indicates a
vector of $k$ zeros.

\begin{definition}
For $s\in\{0,\dots, p\}$, let
\[
    i_{s}, j_{s}\colon \Xpq s{(p-s)} \to \xQp{2p}
\]
be defined by
\begin{align*}
    i_{s}[\vec x:\vec y] &= [\vec 0_{p-s}:\vec x:\vec 0_{s}:\vec y] \\
\intertext{and}
    j_{s}[\vec x:\vec y] &= [\vec x:\vec 0_{p-s}:\vec y:\vec 0_{s}].
\end{align*}
\end{definition}

\begin{proposition}\label{prop:xo2 cofibration}
For $s\in\{0,\dots, p\}$ we have a cofibration
\[
    \Xpq s{(p-s)}_+ \xrightarrow{i_{s}} (\xQp{2p})_+ \to \susp^\nu j_{p-s}(\Xpq {(p-s)}s)_+
\]
where 
\[
    \nu \iso s\omega\dual \dirsum (p-s)\chi\omega\dual - \chi O(2)    
\]
is the normal bundle to the inclusion $j_{p-s}$.
Therefore, we have a long exact sequence
\[
    \cdots\to \HH^{\alpha-\nu}(\Xpq {(p-s)}s) \xrightarrow{(j_{p-s})_!}
    \HH^{\alpha}(\xQp{2p}) \xrightarrow{i_{s}^*}
    \HH^{\alpha}(\Xpq s{(p-s)}) \to\cdots .
\]
\end{proposition}

\begin{proof}
There is a projection
\[
    \pi\colon \xQp{2p} \setminus i_{s}(\Xpq s{(p-s)}) \to j_{p-s}(\Xpq {(p-s)}s)
\]
given by
\[
    \pi[\vec x:\vec y] = [x_{1}:\cdots:x_{p-s}:\vec 0_{s}:y_{p}:\cdots:y_{p-s+1}:\vec 0_{p-s}],
\]
which we can identify with the projection of the normal bundle to $j_{p-s}$.
The result follows.
\end{proof}

Note that, as a grading,
\[
    \nu = s\omega + (p-s)\chi\omega - 2\sigma.
\]

We will use the preceding result to calculate the cohomology of $\xQp{2p}$ by showing that each of these
cofibrations gives a short exact sequence in a range of gradings.
To state the result, we introduce the following notation.

\begin{definition}\label{def:rs m}
For $n\in\Z$, let
\[
    s_n =
    \begin{cases}
        p & \text{if $n\geq p$} \\
        \lfloor (p+n)/2 \rfloor
        & \text{if $-p < n < p$} \\
        0 & \text{if $n\leq -p$.}
    \end{cases}
\]
\end{definition}

\begin{lemma}\label{lem:basis}
Let $n\in\ZZ$. Then the restriction map
\[
    \HH^{n\Omega_1+RO(\GG)}(BU(1)) \to \HH^{n\Omega_1+RO(\GG)}(\Xpq {s_n}{(p-s_n)})
\]
is surjective.
\end{lemma}

\begin{proof}
For notational simplicity, we write $s$ for $s_n$ throughout this proof.

The result follows from the basis for $\HH^\gr(\Xpq s{(p-s)})$ given in \cite{CHTFiniteProjSpace}.
In general, the restriction map of the lemma is not surjective because of the presence in
the cohomology of $\Xpq s{(p-s)}$ of elements like $\zeta_0^{-1}\cwd[s]$ that are not present
in the cohomology of $BU(1)$. However, in the indicated gradings, the bases
involve no such monomials.
For those dedicated readers who want to see the details, these are the bases in question:
\begin{itemize}
    \item If $n \leq -p$, then a basis for $\HH^{n\Omega_1 + RO(\GG)}(\Xq p)$ is
    \[
        \{ \zeta_0^{|n|}, \zeta_0^{|n|-1}\cxwd, \ldots, \zeta_0^{|n|-p+1}\cxwd[p-1] \}.
    \]

    \item If $-p < n < p$ and $p-n$ is even, then $s = (p+n)/2$ and a basis for
    $\HH^{n\Omega_1 + RO(\GG)}(\Xpq s{(p-s)})$ when $n < 0$ is
    \[
        \{ \zeta_0^{|n|}, \zeta_0^{|n|-1}\cxwd, \ldots, \cxwd[|n|], \zeta_0\cwd\cxwd[|n|], \cwd\cxwd[|n|+1],
            \zeta_0\cwd[2]\cxwd[|n|+1], \ldots, \zeta_0\cwd[s]\cxwd[p-s-1] \}
    \]
    and, when $n\geq 0$, is
    \[
        \{ \zeta_1^n, \zeta_1^{n-1}\cwd, \ldots, \cwd[n], \zeta_0\cwd[n+1], \cwd[n+1]\cxwd, \zeta_0\cwd[n+2]\cxwd,
            \ldots, \zeta_0\cwd[s]\cxwd[p-s-1] \}.
    \]

    \item If $-p < n < p$ and $p-n$ is odd, then $s =(p+n-1)/2$ and a basis for
    $\HH^{n\Omega_1 + RO(\GG)}(\Xpq s{(p-s)})$ when $n<0$ is
    \[
        \{ \zeta_0^{|n|}, \zeta_0^{|n|-1}\cxwd, \ldots, \cxwd[|n|], \zeta_0\cwd\cxwd[|n|], \cwd\cxwd[|n|+1],
            \zeta_0\cwd[2]\cxwd[|n|+1], \ldots, \cwd[s]\cxwd[p-s-1] \}
    \]
    and, when $n\geq 0$, is
    \[
        \{ \zeta_1^n, \zeta_1^{n-1}\cwd, \ldots, \cwd[n], \zeta_0\cwd[n+1], \cwd[n+1]\cxwd, \zeta_0\cwd[n+2]\cxwd,
            \ldots, \cwd[s]\cxwd[p-s-1] \}.
    \]

    \item If $n \geq p$, then a basis for $\HH^{n\Omega_1 + RO(\GG)}(\Xp p)$ is
    \[
        \{ \zeta_1^n, \zeta_1^{n-1}\cwd, \ldots, \zeta_1^{n-p+1}\cwd[p-1] \}.
    \]
\end{itemize}
All of these elements are in the image of restriction from $\HH^\gr(BU(1))$.
\end{proof}

\begin{proposition}\label{prop:xo2 splitting}
Let $n\in\Z$, let $\alpha\in n\Omega_1+RO(\GG)$,
and let $s = s_n$. Then we have a split short exact sequence
\[
   0\to \HH^{\alpha-\nu}(\Xpq {(p-s)}s) \xrightarrow{(j_{p-s})_!}
    \HH^{\alpha}(\xQp{2p}) 
    \xrightarrow{i_s^*}
    \HH^{\alpha}(\Xpq s{(p-s)}) \to 0
\]
where $\nu = s\omega + (p-s)\chi\omega - 2\sigma$.
\end{proposition}

\begin{proof}
That $i_{s}^*$ is surjective follows from the preceding lemma.
The splitting occurs because we know from
\cite{CHTFiniteProjSpace} that $\HH^\gr(\Xpq s{(p-s)})$ is a free module over $\HS$.
\end{proof}

This gives the additive structure of the cohomology, which we can write as follows.

\begin{corollary}\label{cor:xo2 additive}
For $n\in\Z$ and $s = s_n$, we have
\begin{multline*}
    \HH^{n\Omega_1+RO(\GG)}(\xQp{2p}) \iso
    \HH^{n\Omega_1+RO(\GG)}(\Xpq s{(p-s)}) \\
    \dirsum 
    \susp^{\nu}\HH^{n\Omega_1 - \nu + RO(\GG)}(\Xpq {(p-s)}s)
\end{multline*}
where $\nu = s\omega + (p-s)\chi\omega - 2\sigma$.
\qed
\end{corollary}

To describe the basis implied by the corollary, we introduce the following elements.

\begin{definition}
For $s\in\{0,\dots,p\}$, let
\[
    m_s = (j_{p-s})_!(1) = [\Xpq {(p-s)}{s}]^* \in \HH^{\nu}(\xQp{2p})
\]
where, as a grading,
\begin{align*}
    \nu
    &= s\omega + (p-s)\chi\omega - 2\sigma \\
    &= (2s-p)\Omega_1 + 2s + 2(p-s-1)\sigma.
\end{align*}
\end{definition}

If we let $n = 2s-p$, then we can check from the definition that $s_n=s$.
Therefore, $m_s$ is indeed one of the basis elements
given by Proposition~\ref{prop:xo2 splitting}.
Since $2s-p$ has the same parity as $p$, 
these elements $m_s$ occur in every other coset $n\Omega_1 + RO(\GG)$
for $-p\leq n \leq p$.

With this notation, we can say that the basis implied by Corollary~\ref{cor:xo2 additive}
in the coset $n\Omega_1 + RO(\GG)$
can be written as a copy of the $p$ elements given by \cite{CHTFiniteProjSpace} as
forming a basis of $\HH^{n\Omega_1+RO(\GG)}(\Xpq {s_n}{(p-s_n)})$,
together with the elements we get by multiplying $m_{s_n}$ by each of the $p$ basis elements
of $\HH^{n\Omega_1 - \nu + RO(\GG)}(\Xpq {(p-s_n)}{s_n})$.

We now describe the multiplicative structure.

\begin{theorem}\label{thm:xo2 mutliplicative}
As an algebra over $\HH^\gr(BU(1))$, the ordinary cohomology $\HH^\gr(\xQp{2p})$ is generated
by the elements $m_s$ with $s\in\{0,\dots,p\}$,
subject to the facts that
$\cwd[p-s] m_s$ is infinitely divisible by $\zeta_0$ and $\cxwd[s] m_s$
is infinitely divisible by $\zeta_1$, and the relations
\begin{alignat*}{3}
   \text{(i)}&\qquad &  \cwd[s]\cxwd[p-s-1] &= \zeta_0 m_{s+1} + \zeta_1 m_{s} &\qquad&\text{for\ \  $0\leq s \leq p-1$;} \\
   \text{(ii)} &&  \cxwd m_{s+1} &= \cwd m_{s} &&\text{for\ \  $0 \leq s \leq p-1$;} \\
   \text{(iii)} && x_{s} x_{p-s} &= 0 &&\text{for\ \  $0\leq s \leq p$.}
\end{alignat*}

\end{theorem}

\begin{proof}
We first check that the relations given do hold.
That $\cwd[p-s]m_s$ is infinitely divisible by $\zeta_0$ follows from the fact that
$\cwd[p-s]m_s = (j_{p-s})_!(\cwd[p-s])$ and $\cwd[p-s]$ is infinitely divisible by $\zeta_0$
in the cohomology of $\Xpq {(p-s)}s$, per \cite{CHTFiniteProjSpace}.
That $\cxwd[s]m_s$ is infinitely divisible by $\zeta_1$ is similar.

We pause to note that we can now say that the fact
that the elements $m_s$ generate follows from Proposition~\ref{prop:xo2 splitting}
and Lemma~\ref{lem:basis}, in the sense that we also need to adjoin elements
like $\zeta_0^{-1}\cwd[p-s]m_s$ that we know exist.

We could give algebraic proofs of the remaining relations, using the bases implied
by Proposition~\ref{prop:xo2 splitting} and determining coefficients by reducing nonequivariantly
and examining fixed points. We find the following geometric verifications more enlightening.
See the Appendix for details on representing cohomology elements using singular manifolds.

To see that $\cwd[s]\cxwd[p-s-1] = \zeta_0 m_{s+1} + \zeta_1 m_{s}$,
represent $\cwd[s]\cxwd[p-s-1]$ by
\begin{multline*}
    \{ [\vec x:\vec y] \in \xQp{2p} \mid x_i = 0 \text{ for } p-s+1\leq i \leq p \text{ and } y_j = 0 \text{ for } 1\leq j\leq p-s-1 \} \\
    = j_{p-s-1}(\Xpq {(p-s-1)}{(s+1)}) \union j_{p-s}(\Xpq{(p-s)}{s}
\end{multline*}
(since the defining equation for $\xQp{2p}$ reduces to $x_{p-s}y_{p-s} = 0$ on this subset)
with intersection we'll write as 
\begin{multline*}
    \Xpq{(p-s-1)}{s} = \{ [\vec x:\vec y] \in \xQp{2p} \mid 
    x_i = 0 \text{ for } p-s\leq i\leq p \\ \text{ and } y_j = 0 \text{ for } 1\leq j\leq p-s \}.
\end{multline*}
We can take the intersection to be the singular part, getting
\begin{align*}
    \cwd[s]\cxwd[p-s-1] &= [j_{s+1}(\Xpq{(p-s-1)}{(s+1)} \union j_{p-s}(\Xpq {(p-s)}s); \Xpq{(p-s-1)}{s}]^* \\
    &= [j_{p-s-1}(\Xpq{(p-s-1)}{(s+1)}; \Xpq{(p-s-1)}{s}]^* \\
    &\qquad {}+ [j_{p-s}(\Xpq {(p-s)}s); \Xpq{(p-s-1)}{s}]^* \\
    &= \zeta_0m_{s+1} + \zeta_1m_{s}.
\end{align*}

To see that $\cxwd m_{s+1} = \cwd m_{s}$, both sides are represented by $\Xpq{(p-s-1)}{s}$ as above.

To see that $m_{s}m_{p-s} = 0$, we know that these two elements are represented by
$j_{p-s}(\Xpq {(p-s)}s)$ and $j_{s}(\Xpq s{(p-s)})$, respectively.
However, we could also represent $m_{p-s}$ by $i_{s}(\Xpq s{(p-s)})$,
because the intersection of the associated affine subspaces 
$j_s(\Cpq s{(p-s)})\intersect i_s(\Cpq s{(p-s)})$ has even codimension
$2\cdot\min\{s,p-s\}$ in each of them.
Since $j_{p-s}(\Xpq {(p-s)}s)$ is disjoint from $i_{s}(\Xpq s{(p-s)})$, the product $m_{s}m_{p-s}$ is 0.

Finally, we need to explain why these relations suffice to determine the whole structure
of the algebra. 
We first note an implication of the relations:
\begin{equation}\label{eqn:xo2 identity}
    \cwd[s]\cxwd[p-s] = (\zeta_0\cwd + \zeta_1\cxwd)m_{s}.
\end{equation}
To see this, if $s < p$, we have
\begin{align*}
    \cwd[s]\cxwd[p-s] &= (\cwd[s]\cxwd[p-s-1])\cxwd \\
    &= \zeta_0\cxwd m_{s+1} + \zeta_1\cxwd m_{s} \\
    &= (\zeta_{0}\cwd + \zeta_1\cxwd) m_{s}
\end{align*}
and, if $s = p$,
\begin{align*}
    \cwd[p] &= \cwd\cdot\cwd[p-1] \\
    &= \zeta_0\cwd m_{p} + \zeta_1\cwd m_{p-1} \\
    &= (\zeta_0\cwd + \zeta_1\cxwd) m_{p}.
\end{align*}
Let 
\[
    A = \HH^\gr(B_G U(1))[m_{s},\ \zeta_0^{-k}\cwd[p-s]m_{s},\ \zeta_1^{-k}\cxwd[s] m_{s} \mid 0\leq s\leq p,\ k\geq 1],
\]
meaning that we take the polynomial algebra in the symbols shown and impose the relations that
\begin{align*}
    \zeta_0\cdot\zeta_0^{-k}\cwd[p-s]m_{s} &= \zeta_0^{-(k-1)}\cwd[p-s]m_{s} \\
\intertext{and}
    \zeta_1\cdot\zeta_1^{-k}\cxwd[s]m_{s} &= \zeta_1^{-(k-1)}\cxwd[s]m_{s}
\end{align*}
for $k\geq 1$. Let $I\subset A$ be the ideal generated by the relations given in the statement of the theorem.
Consider an $n\in\ZZ$ and fix $s = s_n$.
Define ideals $J$ and $K$ by
\[
    J = \rels{m_{p-s},\ \zeta_0^{-k}\cwd[s]m_{p-s},\ \zeta_1^{-k}\cxwd[p-s]m_{p-s} \mid k\geq 1}\subset A
\]
and
\[
    K = \rels{m_{p-s},\ \zeta_0^{-k}\cwd[s]m_{p-s},\ \zeta_1^{-k}\cxwd[p-s]m_{p-s} \mid k\geq 1}\subset B = \HH^\gr(\xQp{2p})
\]
(Morally, $J$ and $K$ are the principal ideals generated by $m_{p-s}$ in their respective
rings, but we need to explicitly include the quotient elements indicated.)
We have a ring map $f\colon A\to B$ taking generators to elements of the same name, and we have shown above
that $I \subset\ker f$, so we get the following commutative diagram.
\begin{equation}\label{eqn:five lemma}
    \begin{gathered}
    \xymatrix{
     0 \ar[r] & (J+I)/I \ar[r] \ar[d] & A/I \ar[r] \ar[d] & A/(J+I) \ar[r] \ar[d] & 0 \\
     0 \ar[r] & K \ar[r] & B \ar[r] & B/K \ar[r] & 0
    }
    \end{gathered}
\end{equation}
It follows from Proposition~\ref{prop:xo2 splitting} that
\begin{align*}
    K^{n\Omega_1+RO(\GG)} &\iso \susp^{\nu} \HH^{n\Omega_1-\nu+RO(\GG)}(\Xpq {p-s}s) \\
\intertext{and}
    (B/K)^{n\Omega_1+RO(\GG)} &\iso \HH^{n\Omega_1+RO(\GG)}(\Xpq s{(p-s)}).
\end{align*}
where $\nu = s\omega + (p-s)\chi\omega - 2\sigma$.
Now, $A/(J+I)$ is obtained by setting $m_{p-s} = 0$ as well as all the terms
$\zeta_0^{-k}\cwd[s]m_{p-s}$ and $\zeta_1^{-k}\cxwd[p-s]m_{p-s}$,
then imposing the relations in the statement of the theorem.
From (\ref{eqn:xo2 identity}) we then get
\[
    \cwd[s]\cxwd[p-s] = (\zeta_0\cwd + \zeta_1\cxwd)m_{s} = 0.
\]
On the other hand, any monomial involving another $m_{t}$ that appears
in grading $n\Omega_1 + RO(\GG)$ must be a multiple of $\zeta_0 m_{t}$, $\zeta_1 m_{t}$,
$\cwd m_{t}$, or $\cxwd m_{t}$, and the relations allow us to rewrite such
monomials in terms of $m_{p-s}$, $\zeta_0$, $\zeta_1$, $\cwd$, and $\cxwd$,
hence as an element coming from $\HH^\gr(BU(1))$.
Finally, from \cite{CHTFiniteProjSpace}, we can see that
$\HH^{n\Omega_1+RO(\GG)}(\Xpq s{(p-s)})$ is obtained from $\HH^{n\Omega_1+RO(\GG)}(BU(1))$
by imposing the one relation $\cwd[s]\cxwd[p-s] = 0$.
It follows that
\[
    A/(J+I) \iso B/K
\]
in gradings $n\Omega_1 + RO(\GG)$.

Now consider $(J+I)/I \iso J/(J\intersect I)$.
One issue we need to deal with is possible elements in $J$ involving products $m_{p-s}m_{t}$.
 To deal with these, assume that $0 < r \leq p-s$ and write
\begin{align*}
    \zeta_0 m_{p-s}m_{s+r} &= m_{p-s}(\cwd[s+r-1]\cxwd[p-s-r] - \zeta_1 m_{s+r-1}) \\
    &= \cwd[s+r-1]\cxwd[p-s-r]m_{p-s} - \zeta_1 m_{p-s}m_{s+r-1}.
\end{align*}
The first term on the right is divisible by $\zeta_0$ by assumption
and the second by induction on $r$
(the beginning of the induction being $m_{p-s}m_{s} = 0$).
By induction, we can then write $m_{p-s}m_{s+r}$ as a multiple of $m_{p-s}$, with the other
factor a polynomial in $\zeta_0$, $\zeta_1$, $\cwd$, and $\cxwd$ 
(with $\zeta_0^{-k}\cwd[s]m_{p-s}$ or $\zeta_1^{-k}\cxwd[p-s]m_{p-s}$ possibly appearing).
A similar argument allows us to rewrite $m_{p-s}m_{s-r}$ for $0 < r \leq s$.
We also have, from (\ref{eqn:xo2 identity}), that
\[
    \cwd[p-s]\cxwd[s] m_{p-s} = (\zeta_0\cwd + \zeta_1\cxwd)m_{s}m_{p-s} = 0.
\]
We then see that we have the generators
and relations needed to show that
\[
    (J+I)/I \iso \susp^{\nu}\HH^{m\Omega_1-\nu+RO(\GG)}(\Xpq {p-s}s).
\]

The two isomorphisms shown then imply, via (\ref{eqn:five lemma}),
that $A/I \iso B$, in other words, the relations given in the statement
of the theorem determine the ring $\HH^\gr(\xQp{2p})$ as claimed.

\end{proof}

\begin{remark}
Equation (\ref{eqn:xo2 identity}) is interesting in its own right.
From \cite[Proposition~6.5]{CHTFiniteProjSpace}, we can rewrite it as
\[
    \cwd[s]\cxwd[p-s] = e(\chi O(2)) m_{s}.
\]
That the  Euler class of $\chi O(2)$ shows up here can be explained using an equivariant analogue
of the excess intersection formula, as given in the nonequivariant case
by Fulton \cite[Theorem~6.3]{FultonIntersection}.
The class $\cwd[s]\cxwd[p-s]$ is represented by $\Xpq {(p-s)}s$ in the cohomology of $\Xpq pp$.
The intersection with $\xQp{2p}$ should then represent $\cwd[s]\cxwd[p-s]$ in the cohomology of $\xQp{2p}$,
except that the intersection is not transverse, it is all of $\Xpq {(p-s)}s$.
Where the sum of the two tangent bundles on the intersection should be the whole tangent bundle of $\Xpq pp$, it
falls short by the ``excess normal bundle'' $\chi O(2)$, the normal bundle to $\xQp{2p}$ in $\Xpq pp$.
The excess intersection formula then tells us that, in the cohomology of $\xQp{2p}$,
$\cwd[s]\cxwd[p-s]$ is the product of the intersection, $[\Xpq {(p-s)}s]^*$, with the Euler class
of the excess normal bundle, that is,
$\cwd[s]\cxwd[p-s] = e(\chi O(2)) [\Xpq {(p-s)}s]^* = e(\chi O(2)) m_{p-s}$.
We will give a proof of this equivariant excess intersection formula in \cite{CH:QuadricsII},
where we will need it for other calculations.
\end{remark}

It's interesting to see how the structure in Theorem~\ref{thm:xo2 mutliplicative} restricts to the $RO(\GG)$-graded part.

\begin{corollary}\label{cor:ro(g) basis}
Let $s = \lfloor p/2 \rfloor$.
If $p$ is even, then $\HH^{RO(\GG)}(\xQp{2p})$ has a basis over $\HS$ given by
\begin{multline*}
    \{ 1,\ \zeta_0\cwd,\ \cwd\cxwd,\ \zeta_0\cwd[2]\cxwd,\ \cwd[2]\cxwd[2], \ldots,\ \zeta_0\cwd[s]\cxwd[s-1], \\
        m_{s},\ \zeta_0\cwd m_{s},\ \cwd\cxwd m_{s},\ \zeta_0\cwd[2]\cxwd m_{s},
        \ \cwd[2]\cxwd[2] m_{s}, \ldots,\ \zeta_0\cwd[s]\cxwd[s-1] m_{s} \}.
\end{multline*}
If $p$ is odd, then a basis is given by
\begin{multline*}
      \{ 1,\ \zeta_0\cwd,\ \cwd\cxwd,\ \zeta_0\cwd[2]\cxwd,\ \cwd[2]\cxwd[2],\ \ldots,\ \cwd[s]\cxwd[s], \\
        \zeta_1 m_{s},\ \cwd m_{s},\ \zeta_0\cwd[2] m_{s},\ \cwd[2]\cxwd m_{s}, \\
        \zeta_0\cwd[3]\cxwd m_{s},\ \ldots,\ \zeta_0\cwd[s+1]\cxwd[s-1] m_{s} \}.
\end{multline*}
\qed
\end{corollary}

It's useful to picture how these basis elements are arranged. The following diagrams
show the locations of the basis elements in the $RO(\GG)$ grading for $p = 4$ and $p=5$.
A basis element in grading $a + b\sigma$ is shown at $(a,b)$, and the grid lines
are spaced every two.
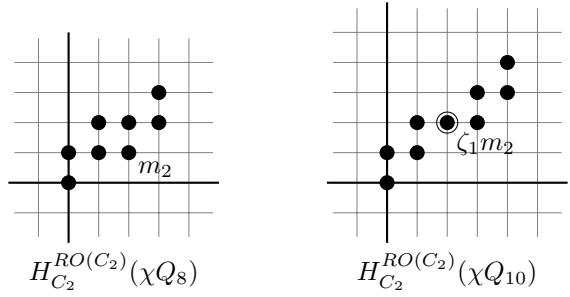
\[
\begin{tikzpicture}[scale=0.4]
	\draw[step=1cm, gray, very thin] (-1.8, -1.8) grid (4.8,4.8);
	\draw[thick] (-2, 0) -- (5, 0);
	\draw[thick] (0, -2) -- (0, 5);

 	\node[below] at (1.5, -2) {$\HH^{RO(\GG)}(\xQp{8})$};

    \fill (0,0) circle(0.25cm);
    \fill (0,1) circle(0.25cm);
    \fill (1,1) circle(0.25cm);
    \fill (1,2) circle(0.25cm);
    \fill (2,1) circle(0.25cm);
    \fill (2,2) circle(0.25cm);
    \fill (3,2) circle(0.25cm);
    \fill (3,3) circle(0.25cm);
    \node[below right] at (2,1) {$m_{2}$};

\end{tikzpicture}
\qquad\qquad
\begin{tikzpicture}[scale=0.4]
	\draw[step=1cm, gray, very thin] (-1.8, -1.8) grid (5.8,5.8);
	\draw[thick] (-2, 0) -- (6, 0);
	\draw[thick] (0, -2) -- (0, 6);

 	\node[below] at (2, -2) {$\HH^{RO(\GG)}(\xQp{10})$};

    \fill (0,0) circle(0.25cm);
    \fill (0,1) circle(0.25cm);
    \fill (1,1) circle(0.25cm);
    \fill (1,2) circle(0.25cm);
    \fill (2,2) circle(0.25cm); \draw (2,2) circle(0.35cm);
    \fill (3,2) circle(0.25cm);
    \fill (3,3) circle(0.25cm);
    \fill (4,3) circle(0.25cm);
    \fill (4,4) circle(0.25cm);
    \node[below right] at (2,2) {$\zeta_1 m_{2}$};

\end{tikzpicture}
\]
In the first diagram, $m_{2}$ lies in grading $4 + 2\sigma$.
The double circle in the second diagram indicates that there are two basis elements
in that grading, $4 + 4\sigma$, which the corollary
gives as $\cwd[2]\cxwd[2]$ and $\zeta_1 m_{2}$.
Since we have $\cwd[2]\cxwd[2] = \zeta_0 m_{3} + \zeta_1 m_{2}$,
we could instead use $\zeta_0 m_{3}$ and $\zeta_1 m_{2}$ as the basis
elements in that grading,
similar to the nonequivariant choice between the bases $\{ \cd[p-1],m_{[0]} \}$ or
$\{ m_{[0]}, m_{[1]} \}$.
In the case of $\xQp 8$, the basis element in grading $2 + 4\sigma$ given by
Corollary~\ref{cor:ro(g) basis} is $\zeta_0\cwd[2]\cxwd$,
which can be rewritten as
\[
    \zeta_0\cwd[2]\cxwd = \zeta_0^2 m_{3} + \xi m_{2},
\]
so we could use $\zeta_0^2 m_{3}$ as a basis element in place of $\zeta_0\cwd[2]\cxwd$,
but this feels less natural than the replacement suggested above in the case of $\xQp{10}$.

To describe the multiplicative structure of the $RO(\GG)$-graded part
purely in terms of these bases is messy.
However, calculations are easily done by viewing these elements as the products they are
and using the relations of Theorem~\ref{thm:xo2 mutliplicative}.
Notice in particular that, if $p$ is even, then $m_{s}^2 = 0$.
If $p$ is odd, we carry out a computation alluded to in the proof of the theorem:
\begin{align*}
    \zeta_1 m_{s+1}^2 &= m_{s+1}(\cwd[s]\cxwd[s] - \zeta_0 m_{s}) \\
        &= \cwd[s]\cxwd[s]m_{s+1},
\end{align*}
so
\[
    m_{s+1}^2 = \zeta_1^{-1}\cwd[s]\cxwd[s]m_{s+1}
\]
using the fact that $\cxwd[s]m_{s+1}$ is divisible by $\zeta_1$.
Although this is a calculation outside of the $RO(\GG)$ grading,
it can be used in the process of doing calculations in that grading.
For example, we can compute
\begin{align*}
    (\zeta_1 m_{s+1})^2 &= \zeta_1^2\cdot \zeta_1^{-1}\cwd[s]\cxwd[s]m_{s+1} \\
    &= \zeta_1\cwd[s]\cxwd[s]m_{s+1} \\
    &= \cwd[s]\cxwd[s-1]m_{s+1}\bigl((1-\kappa)\zeta_0\cwd + e^2\bigr) \\
    &= (1-\kappa)\zeta_0\cwd[s+1]\cxwd[s-1]m_{s+1} + e^2\cwd[s]\cxwd[s-1]m_{s+1},
\end{align*}
writing the answer in terms of the basis given in Corollary~\ref{cor:ro(g) basis}.

\section{Restriction maps}\label{sec:restrictions}

For computations, and to compare our result to nonequivariant calculations,
it is useful to know the maps
\begin{align*}
    \rho\colon \HH^\alpha(\xQexp p) &\to H^{\rho(\alpha)}(\Qexp{2p}), \\
\intertext{restriction to nonequivariant cohomology, and}
    (-)^\GG\colon \HH^\alpha(\xQexp p) &\to H^{\alpha^\GG}(\xQexp p^\GG) \\
        &\qquad\iso H^{\alpha_0^\GG}(\Xp p) \dirsum H^{\alpha_1^\GG}(\Xq p),
\end{align*}
the fixed-point map.
The map $\rho\colon RO(\Pi \xQp{2p})\to \Z$ on gradings is the
forgetful map that gives the nonequivariant dimension of the representation.
In the case of the fixed-point map, $\alpha^\GG = (\alpha_0^\GG,\alpha_1^\GG) \in\Z^2$ gives
the dimensions of the fixed points of the representation $\alpha$ on each
of the two components of $\xQp{2p}^\GG$.
(An alternative approach, mentioned in the Appendix,
is to regrade the nonequivariant cohomologies on $RO(\Pi\xQp{2p})$
via $\rho$ or $(-)^\GG$, but it is simpler for computations to use the integer
grading. See \cite{CH:bu2} for a general discussion on regrading graded rings.)

For $\rho$, we have the following.

\begin{proposition}
The restriction map $\rho$ is given by the values
\begin{align*}
    \rho(\zeta_0) &= 1 & \rho(\zeta_1) &= 1 \\
    \rho(\cwd) &= \cd & \rho(\cxwd) &= \cd \\
    \text{and\ \ }\rho(m_{s}) &=
            m_{[s]}.
\end{align*}
\end{proposition}

\begin{proof}
The calculations of $\rho(\zeta_0)$, $\rho(\zeta_1)$, $\rho(\cwd)$, and $\rho(\cxwd)$ are inherited
from the same calculations for $BU(1)$.
The calculation of $\rho(m_{s})$ follows from the discussion in the Introduction.
\end{proof}

In particular, we can compare the relations given
in Theorem~\ref{thm:xo2 mutliplicative} to the nonequivariant relations discussed
in the Introduction. When we apply $\rho$, we get that
\begin{alignat*}{3}
    \cwd[s]\cxwd[p-s-1] &= \zeta_0 m_{s+1} + \zeta_1 m_{s} &\quad\text{becomes}\quad&&
        \cd[p-1] &=m_{[s+1]}+m_{[s]}=m_{[0]}  + m_{[1]}, \\
    \cxwd m_{s+1} &= \cwd m_{s} &\quad\text{becomes}\quad&&
        \cd\, m_{[s+1]} &= \cd\, m_{[s]}, 
        \\
    m_{s}m_{p-s} &= 0 &\quad\text{implies\ }\quad&&
       m_{[0]}\,m_{[1]} &=
    \begin{cases} 
        \hfill 0 \hfill & \text{if $p$ is odd}\\
        \cd[p-1]m_{[0]}&\text{if $p$ is even.}
    \end{cases}    
\end{alignat*}
For the last, we are using that
\[
    0 = \rho(m_{s}m_{p-s}) =
    \begin{cases}
        m_{[0]}^2 & \text{if $p$ and $s$ are even} \\
        m_{[1]}^2 & \text{if $p$ is even and $s$ is odd} \\
        m_{[0]}m_{[1]} & \text{if $p$ is odd,}
    \end{cases}
\]
together with the fact that 
\[m_{[0]}^2+m_{[0]}m_{[1]}=\cd[p-1]m_{[0]}=\cd[p-1]m_{[1]}=m_{[0]}m_{[1]}+m_{[1]}^2,
\]
which is obtained by multiplying the first nonequivariant relation by both $m_{[0]}$ and $m_{[1]}$.
\begin{proposition}
The fixed-point map is given by the values
\begin{align*}
    \zeta_0^\GG &= (0,1) & \zeta_1^\GG &= (1,0) \\
    \cwd[\GG] &= (\cd,1) & \cxwd[\GG] &= (1,\cd) \\
    m_{s}^\GG &= (\cd[s],\cd[p-s]).
\end{align*}
\end{proposition}

\begin{proof}
Again, the first four values are known from the cohomology of $BU(1)$.
For the last, $m_{s}^\GG$ is represented by 
\[
    j_{p-s}(\Xpq {(p-s)}s)^\GG = \Xp {p-s} \disjunion \Xq s,
\]
and $\Xp {p-s}$ represents the power $\cd[s] \in H^*(\Xp p)$ while $\Xq s$ represents $\cd[p-s]\in H^*(\Xq p)$.
\end{proof}

As with $\rho$, it's interesting to see what happens to each of the relations when we take fixed points.
\begin{alignat*}{3}
    \cwd[s]\cxwd[p-s-1] &= \zeta_0 m_{s+1} + \zeta_1 m_{s} &\ \text{becomes}\ &&
        (\cd[s],\cd[p-s-1]) &= (0,\cd[p-s-1]) + (\cd[s],0)\\
    \cxwd m_{s+1} &= \cwd m_{s} &\ \text{becomes}\ &&
        (\cd[r+1],\cd[s]) &= (\cd[r+1],\cd[s]) \\
    m_{s}m_{p-s} &= 0 &\ \text{becomes}\ &&
        (\cd[s],\cd[p-s])(\cd[p-s],\cd[s]) &= (\cd[p],\cd[p]) = (0,0).
\end{alignat*}
The triviality of the first two reflects the fact that the only interesting relation
in the nonequivariant cohomology of $\Xp p$ is that $\cd[p] = 0$.

\section{A Grassmannian}\label{sec:grassmannian}

We look at an interesting special case in detail:
The Grassmannian $\Grpq 23{}$ of $2$-planes in $\Cpq 3{}$ is $\GG$-diffeomorphic to
$\xQexp{3}$. 
To see this, we use a modified version of the Pl\"ucker embedding.
We represent 2-planes in $\Cpq 3{}$ by $2\times 4$-matrices of complex numbers, which we
think of as arrays $[\vec a_1\ \vec a_2\ \vec a_3\ \vec a_4]$ with linearly independent rows, in which
the columns
$\vec a_1$, $\vec a_2$, and $\vec a_3$ lie in $\Cp 2$ while $\vec a_4$ is in $\Cq 2$.
Such a matrix represents the 2-plane in $\Cpq 3{}$ spanned by its rows, and two such matrices represent the
same 2-plane if they are in the same orbit of the general linear group $\GL(2,\C)$.
The Pl\"ucker embedding (modified slighly by a negative sign) is then the map
$\psi\colon \Grpq 23{}\to \Xpq 33$ defined by
\[
    \psi[\vec a_1\ \vec a_2\ \vec a_3\ \vec a_4]
    = [\Delta_{12}:\Delta_{13}:\Delta_{23}:\Delta_{14}:-\Delta_{24}:\Delta_{34}],
\]
where $\Delta_{ij} = |\vec a_i\ \vec a_j|$ is the determinant.
The image of this map is exactly $\xQexp 3$.
The fixed set $\Grpq 23{}^\GG$ consists of two components, $\Grp 23$, which maps to
$\Xp 3\subset \xQexp 3$, and $\Xp 3\times\Xq {}$, which maps to $\Xq 3\subset \xQexp 3$.
This identifies the representation rings $RO(\Pi\Grpq 23{})$ and $RO(\Pi\xQexp 3)$ in a natural way.

The Grassmannian has a tautological 2-plane bundle $\pi$.
The Pl\"ucker embedding classifies its exterior power, or determinant line bundle, $\lambda = \ext^2\pi$,
that is, $\lambda = \psi^*\omega$.
Writing $\cld$ for the Euler class of the dual of $\lambda$, 
and using the structure of $\HH^\gr(BU(1))$ from \cite{Co:InfinitePublished},
Theorem~\ref{thm:xo2 mutliplicative} has the following corollary.

\begin{corollary}
As an algebra over $\HS$, the cohomology $\HH^\gr(\Grpq 23{})$ is generated by elements
\begin{itemize}
    \item $\zeta_0$ in grading $\Omega_0$,
    \item $\zeta_1$ in grading $\Omega_1$,
    \item $\cld$ in grading $\lambda$, where $\lambda = 2 + \Omega_1$,
    \item $\cxld$ in grading $\chi\lambda$, where $\chi\lambda = 2 + \Omega_0$,
    \item $m_{0}$ in grading $-3\Omega_1 + 4\sigma$,
    \item $m_{1}$ in grading $-\Omega_1 + 2 + 2\sigma$,
    \item $m_{2}$ in grading $\Omega_1 + 4$, and
    \item $m_{3}$ in grading $3\Omega_1 + 6 - 2\sigma$,
\end{itemize}
subject to the facts that
\begin{itemize}
    \item $m_{3}$, $\cld m_{2}$, $\cld[2] m_{1}$, and $\cld[3] m_{0}$ are infinitely divisible by $\zeta_0$,
    \item $\cxld[3]m_{3}$, $\cxld[2]m_{2}$, $\cxld m_{1}$, and $m_{0}$ are infinitely divisible by $\zeta_1$,
\end{itemize}
and the relations
\begin{align*}
    \zeta_0\zeta_1 &= \xi \\
    \zeta_1\cxld &= (1-\kappa)\zeta_0\cld + e^2 \\
    \cld[s]\cxld[2-s] &= \zeta_0 m_{s+1} + \zeta_1 m_{s} && 0\leq s \leq 2 \\
    \cxld m_{s+1} &= \cld m_{s} && 0\leq s \leq 2 \\
    m_{3}m_{0} &= 0 \\
    m_{2}m_{1} &= 0.
\end{align*}
\qed
\end{corollary}

\begin{remark}
One question this raises is: What submanifolds of $\Grpq 23{}$ represent the elements $m_{s}$?
Examining the Pl\"ucker embedding, we can verify that
\begin{align*}
    m_{0} &= [\Grp 23]^* \\
    m_{1} &= [\Xp{}\times\Xpq 2{}]^* \\
    m_{2} &= [\Grpq 22{}]^* \\
    m_{3} &= [\Xp 3\times\Xq{}]^*.
\end{align*}
\end{remark}

\begin{remark}
Another question we can ask is: What are the Euler classes $\cgd = e(\pi\dual)$ and $\cxgd = e(\chi\pi\dual)$
in terms of the generators given by the corollary?
(These are the images of the elements we called $\cwd$ and $\cxwd$ in \cite{CH:bu2}.)
To answer, we note that
$\pi\dual$ has a section whose zero locus is
$\Grpq 22{}$, which represents $m_{2}$. The dimensions are correct to give us
\[
    \cgd = m_{2}.
\]

The bundle $\chi\pi\dual$ has a section whose zero locus is $\Grp 23$,
which represents $m_{0}$. However, this element lies in the wrong grading to
be the Euler class of $\chi\pi\dual$. The problem is in the equivariant
dimension we attach to $\Grp 23$, which could be considered to have any
dimension of the form $4 + n\Omega_1$, because its intersection with
the second component of $\Grpq 23{}^\GG$ is empty.
Reinterpreting the dimension in this way amounts to multiplying by $\zeta_1^n$,
and correcting the dimension leads to
\[
    \cxgd = \zeta_1^2 m_{0}.
\]
The factor of $\zeta_1$ is relatively harmless because $m_{0}$ is infinitely divisible
by $\zeta_1$.

With these identifications, we can recover a relation from \cite{CH:bu2}, which here takes the form
\[
    \cxgd = (1-\kappa)\zeta_0^2 \cgd + e^2 \cxld.
\]
We can derive this as follows:
\begin{align*}
    \cxgd &= \zeta_1^2 m_{0} \\
    &= \zeta_1(\cxld[2] - \zeta_0 m_{1}) \\
    &= \zeta_1\cxld[2] - \xi m_{1} \\
    &= (1-\kappa)\zeta_0\cld\cxld + e^2\cxld - (1-\kappa)\xi m_{1} \mathrlap{\qquad (\text{using that }\kappa\xi = 0)} \\
    &= (1-\kappa)\zeta_0\cld\cxld + e^2\cxld - (1-\kappa)\zeta_0(\cld\cxld - \zeta_0 m_{2}) \\
    &= (1-\kappa)\zeta_0^2 m_{2} + e^2\cxld \\
    &= (1-\kappa)\zeta_0^2\cgd + e^2\cxld.
\end{align*}
\end{remark}

\section{27 lines}\label{sec:27 lines}

A well-known example in enumerative geometry is the result that
a cubic surface in $\Xp 4$ contains 27 lines, the set of which we will refer to as $T$;
see, for example \cite[\S6.2.1]{3264EisenbudHarris}.
In \cite{Braz:equivenumerative}, Brazelton determined the structure of $T$ as an $S_4$-set when the symmetric group $S_4$ acts on $\Cp 4$ as its regular representation,
which then implies results for the subgroups of $S_4$.
Here, as an application of our results,
we use the calculation of the cohomology of $\Grpq 23{}$ in the preceding section to not only reproduce one of
Brazelton's calculations in the case of $\GG$, but to determine a bit more information about
those 27 lines.

So we consider the question of how many lines lie on a smooth $\GG$-invariant cubic surface in
$\Xpq 3{}$.
Such a surface is the zero locus of a homogeneous cubic polynomial $f(a_1,a_2,a_3,b)$
with $(a_1,a_2,a_3,b)\in \Cpq 3{}$.
For the zero locus to be $\GG$-invariant, we want either $tf = f$ or $tf = -f$.
However, if $tf = -f$, then $f$ must have $b$ as a factor,
so its zero locus contains the whole hyperplane $b=0$, hence infinitely many lines.
The interesting case is therefore when $tf = f$.
(Brazelton gives an explicit example of an $S_4$-invariant cubic containing exactly 27 lines,
so the problem is not completely vacuous.)

Such a cubic polynomial is a $\GG$-equivariant linear functional
\[
    f\colon \Sym^3(\Cpq 3{}) \to \C
\]
and, on restricting to 2-planes, we get a section of $\Sym^3(\pi\dual)$ over $\Grpq 23{}$.
Each point in the zero locus of this section gives us a line on the cubic surface defined by $f$,
and this zero locus represents the Euler class of the bundle.
Hence, determining the set of such lines comes down to computing the Euler class
$e(\Sym^3(\pi\dual))$ in the cohomology of $\Grpq 23{}$.

In order to compute $e(\Sym^3(\pi\dual))$, we start by determining the grading in which it lives,
which will be the same as the equivariant dimension of $\Sym^3(\pi\dual)$.
For this we examine the fibers over each of the components of $\Grpq 23{}^\GG$.
Over $\Grp 23$, the fibers have the form
\[
    \Sym^3(\Cp 2) \iso \Cp 4.
\]
Over $\Xp 3\times \Xq{}$, the fibers have the form
\[
    \Sym^3(\Cpq 1{}) \iso \Cpq 22.
\]
This gives us
\[
    \grad e(\Sym^3(\pi\dual)) = \dim \Sym^3(\pi\dual)
    =  2\Omega_1 + 8.
\]
From our computations, a basis for $\HH^{2\Omega_1 + RO(\GG)}(\Grpq 23{})$ 
as a module over $\HS$ consists of the following elements.
\begingroup
\allowdisplaybreaks[0]
\begin{alignat*}{2}
    \text{\textbf{Element}} &\qquad& &\text{\textbf{Grading}} \\
    \zeta_1^2 &&& 2\Omega_1 \\
    \zeta_1\cld &&& 2\Omega_1 + 2 \\
    \cld[2] &&& 2\Omega_1 + 4 \\
    \zeta_1 m_{2} &&& 2\Omega_1 + 4 \\
    \cld m_{2} &&& 2\Omega_1 + 6 \\
    \zeta_0^{-1}\cld\cxld m_{2} &&& 2\Omega_1 + 8
\end{alignat*}
\endgroup
From the structure of $\HS$, we see that
the only one of these that can contribute to grading $2\Omega_1 + 8$ is the last, so we must have
\[
    e(\Sym^3(\pi\dual)) = \alpha \zeta_0^{-1}\cld\cxld m_{2}
\]
for some $\alpha \in A(\GG)$.
To determine $\alpha$, we apply $\rho$ and take fixed points.
From the nonequivariant calculation (which is an easy application of the splitting principle), we know that
\[
    \rho(e(\Sym^3(\pi\dual))) = 27\cd[2]_1 \cd_2 = 27\cd[2] m_{[0]},
\]
where we identify $\cd$ and $m_{[0]}$ from the cohomology of $\Qexp 6$ with
$\cd_1 = c_1(\pi\dual)$ and $\cd_2 = c_2(\pi\dual)$, respectively, in the cohomology of $\Grp 24$.
But
\[
    \rho(\alpha \zeta_0^{-1}\cld\cxld m_{2}) = \rho(\alpha) \cd[2] m_{[0]},
\]
hence $\rho(\alpha) = 27$.

To determine $e(\Sym^3(\pi\dual))^\GG = e(\Sym^3(\pi\dual)^\GG)$,
we first look at the restriction of $\Sym^3(\pi\dual)$ to $\Grp 23 \subset \Grpq 23{}^\GG$,
where the action of $\GG$ on the bundle is trivial, hence
\[
    \Sym^3(\pi\dual)^\GG | \Grp 23 = \Sym^3(\pi\dual | \Grp 23)
\]
and its Euler class
is the restriction of its Euler class from $\Grp 24$, $27c_1^2c_2$,
which is 0 in the cohomology of $\Grp 23$.

When we restrict $\pi\dual$ to the other component, $\Xp 3\times\Xq{}\subset \Grpq 23{}^\GG$, it splits as
$\eta\dirsum\Cq{}$, where $\eta$ is the dual of the tautological line bundle over $\Xp 3$.
Hence the restriction of the symmetric power is
\[
    \Sym^3(\pi\dual)|\Xp 3\times\Xq{} \iso 
    \eta^{\tensor 3} \dirsum (\eta^{\tensor 2}\tensor \Cq{})
    \dirsum \eta \dirsum \Cq{}
\]
and its fixed-set bundle is
\[
    \Sym^3(\pi\dual)^\GG|\Xp 3\times\Xq{} \iso
    \eta^{\tensor 3} \dirsum \eta,
\]
whose (nonequivariant) Euler class is $3\cd\cdot \cd = 3\cd[2]$.

So we've calculated
\[
    e(\Sym^3(\pi\dual))^\GG = (0, 3\cd[2]).
\]
Now
\[
    (\alpha \zeta_0^{-1}\cld\cxld m_2)^\GG = \alpha^\GG (0, \cd[2]),
\]
so $\alpha^\GG = 3$. 
Putting that together with $\rho(\alpha) = 27$, we see that
\[
    \alpha = 3 + 12g,
\]
so
\begin{align*}
    e(\Sym^3(\pi\dual)) &= (3 + 12g)\zeta_0^{-1}\cld\cxld m_2 \\
    &= 12[\GG]^* + 3[\Xp{}\times\Xq{}]^*
\end{align*}
where the point $\Xp{}\times\Xq{} \in \Grpq 23{}$ represents $\zeta_0^{-1}\cld\cxld m_{2}$.
That is, the Euler class $e(\Sym^3(\pi\dual))$
is represented by 3 fixed points in the component $\Xp 3\times\Xq{}$
plus 12 free $\GG$-orbits in $\Grpq 23{}$. 

Now we can answer question (\ref{27 lines question}) that we posed in the introduction, and express $T$
as a sum of the fundamental classes of the four types of equivariant lines displayed in Figure \ref{fig:lines}:
$$[T]^*=0[\Xp 2]^*+3[\Xpq 1{}]^*+0[\Xq 2]^*+12[C_2\times \Xp 2]^*.$$
So, of the 27 lines on an even cubic in $\Xpq 3{}$, 24 form 12 pairs of lines interchanged by
the action of $\GG$, while three are $\GG$-invariant lines of the form $\Xpq 1{}$,
that is, while $\GG$-invariant, they have nontrivial $\GG$-action.

We can also say that, considered just as a $\GG$-set of points from $\Grpq 23{}$, the set $T$ is
\[
    12[\GG/e] + 3[\GG/\GG],
\]
but we have more information than that, because we know
in which component of $\Grpq 23{}^\GG$ the points corresponding to the three $\GG$-invariant lines lie,
hence how $\GG$ acts on those lines.

\subsection{Comparison with Brazelton’s result}
This gives the same $\GG$-set as found by Brazelton in \cite{Braz:equivenumerative},
for the action of the subgroup he calls $C_2^o$, which is $\{e,(1\ 2)\} < S_4$.
Note that his action of $C_2^o$ on $\Cp 4$ fixes two summands and interchanges
the other two. However, a simple change of basis identifies this with
the $\GG$ representation $\Cpq 3{}$.

In his Table~1, Brazelton shows the resulting $C_2^o$-set
of lines on the cubic as $12[\GG/e] + 3[\GG/\GG]$ as above, but does not identify
the nature of the three $C_2^o$-invariant lines.
Brazelton's approach is roughly to show that any equivariant cohomology theory
can be extended so that the cohomology of $X$ is graded on $KO_G(X)$ and that doing so then provides
Euler classes for some complex vector bundles when the cohomology theory
is \emph{complex oriented}, such classes a priori depending on a choice of section. 
He also defines pushforwards and uses these to
define Euler numbers, which are elements
of the Burnside ring $A(G)$, and shows their invariance with respect to choice of the section
used to define the Euler class.
By explicit example he finds the action of $S_4$ on a set of 27 lines on an
$S_4$-invariant cubic, and concludes from the invariance of Euler numbers
that any cubic will have the same $S_4$-set of lines.

Brazelton's approach does not apply directly to ordinary cohomology because,
as he points out, it is not a complex oriented theory in the sense in which
that term is usually used.
The extended grading we use is, however, closely related to the one that Brazelton uses.
His approach could be applied to $RO(G)$-graded ordinary cohomology to give a
$KO_G(X)$-graded theory. This maps to our version via the map
$KO_G(X)\to RO(\Pi X)$ that gives the equivariant dimension of a bundle;
this map is generally not surjective, so
ordinary cohomology can have a larger grading than that approach provides.
More important is that ordinary cohomology has the full machinery of Thom isomorphisms
and Poincar\'e duality, so has Euler classes and pushforwards in more
generality. 

To see then how our calculation of the Euler class relates
to Brazelton's calcluation of an Euler number,
we note that, in discussing how the Euler class $e(\Sym^3(\pi\dual))$ is represented,
we are really looking at its Poincar\'e dual in homology.
For convenience of interpretation, we would like that dual to be in the 0th homology group,
but the grading is not quite right. This can be easily fixed:
The class $e(\Sym^3(\pi\dual))$, being a multiple of $\cld m_{2}$,
is divisible by $\zeta_0$, so multiplying or dividing by $\zeta_0$ is
relatively harmless---it amounts to reinterpreting the equivariant dimension of
the representing manifold.
In identifying the manifold, really collection of points, representing $e(\Sym^3(\pi\dual))$, we can
therefore think instead of
the Poincar\'e dual of $\zeta_0^2 e(\Sym^3(\pi\dual))$, which lies in $H^\GG_0(\Grpq 23{})$, 
the group generated by the monoid of
homotopy classes of $\GG$-sets mapping to $\Grpq 23{}$, with addition being disjoint union.
This is where we identified the Euler class as corresponding to the set 
$12[\GG] + 3[\Xp{}\times\Xq{}]$ mapping to $\Grpq 23{}$.
Projecting $\Grpq 23{}$ to a point induces the map $H^\GG_0(\Grpq 23{}) \to A(\GG)$
giving the underlying $\GG$-set and forgetting the map into $\Grpq 23{}$,
and this recovers the Euler number that Brazelton calculated
as $12[\GG/e] + 3[\GG/\GG]$.

\appendix

\section{Equivariant ordinary cohomology}

\subsection{Ordinary cohomology}\label{app:ordinarycohomology}
In this paper
we use $\GG$-equivariant ordinary cohomology with the extended grading developed in \cite{CostenobleWanerBook}.
This is Bredon's ordinary cohomology extended to be graded on representations of the fundamental groupoids
of $\GG$-spaces.
We review here some of the notation and computations we use.
More detailed descriptions of this theory can be found in \cite{CHTFiniteProjSpace}
and \cite{Beaudry:Guide}.

For a $\GG$-space $Y$ over $X$, we write $H_\GG^{RO(\Pi X)}(Y;\Mackey T)$ for the
ordinary cohomology of $Y$ with coefficients in a Mackey functor $\Mackey T$, graded
on $RO(\Pi X)$, the representation ring of the fundamental groupoid of $X$.
Throughout this paper we use the Burnside ring Mackey functor $\Mackey A$ as the coefficients,
and write simply $H_\GG^{RO(\Pi X)}(Y)$.
Note also that, for simplicity of notation, we use unreduced cohomology where in other
papers we have used reduced cohomology.

In \cite{CostenobleWanerBook} and \cite{CHTFiniteProjSpace} we considered cohomology itself to be
Mackey functor--valued, and this will be an important point of view again in \cite{CH:QuadricsII}.
There are several ways of thinking of the Mackey functor--valued theory
$\Mackey H_\GG^{RO(\Pi X)}(Y)$. We define
\[
    \Mackey H_\GG^{RO(\Pi X)}(Y)(\GG/H) = \HH^{RO(\Pi X)}(\GG/H\times Y)
\]
for $H = \GG$ or the trivial subgroup $e$. Considering $r\colon \GG/e\times Y\to Y$, restriction is then
the induced map
\[
    \rho = r^*\colon \Mackey H_\GG^{RO(\Pi X)}(Y)(\GG/\GG) \to \Mackey H_\GG^{RO(\Pi X)}(Y)(\GG/e),
\]
while
\[
    \tau = r_!\colon \Mackey H_\GG^{RO(\Pi X)}(Y)(\GG/e) \to \Mackey H_\GG^{RO(\Pi X)}(Y)(\GG/\GG)
\]
can be thought of, depending on your predilictions, 
as either the transfer map or the pushforward map associated with $r$.
But this hides somewhat the simplicity of the situation, because we have
\[
    \Mackey H_\GG^{\alpha}(Y)(\GG/e) = \HH^\alpha(\GG/e\times Y) \iso H^{|\alpha|}(Y;\Z),
\]
where $|\alpha|$ is the integer rank of the representation $\alpha$.
So $\Mackey H_\GG^{RO(\Pi X)}(Y)(\GG/e)$ is simply the nonequivariant cohomology of $Y$
regraded on $RO(\Pi X)$ via the rank function $|{-}|\colon RO(\Pi X)\to \Z$.

In this paper we concentrate on the groups $\HH^{RO(\Pi X)}(Y)$ at level $\GG/\GG$.
However, we will use the maps $\rho$ and $\tau$ extensively,
thinking of them as connecting the equivariant and nonequivariant cohomologies,
rather than as the structure maps in the Mackey functor--valued theory.
This is simply a matter of emphasis.


For all $X$ and $Y$, $H_\GG^{RO(\Pi X)}(Y)$ is a graded module over 
\[
    \HS =  H_\GG^{RO(\GG)}(S^0),
\]
the cohomology of a point.
The grading on the latter is just $RO(\GG)$, the real representation ring of $\GG$, which is free abelian
on $1$, the class of the trivial representation $\R$, and $\sigma$, the class of the
sign representation $\R^\sigma$. The cohomology of a point was calculated by Stong in an unpublished
manuscript and first published by Lewis in \cite{LewisCP}. We can picture the calculation as 
in Figure~\ref{fig:cohompt},
in which a group in grading $a+b\sigma$ is plotted at the point $(a,b)$, and the spacing of the grid lines
is 2 (which is more conventient for other graphs we will give).
The square box at the origin is a copy of $A(\GG)$, the Burnside ring of $\GG$,
closed circles are copies of $\Z$, and open circles are copies of $\Z/2$.
\begin{figure}
\qquad\qquad\qquad
\begin{tikzpicture}[x=4mm, y=4mm]
	\draw[step=2, gray, very thin] (-7.8, -7.8) grid (7.8, 7.8);
	\draw[thick] (-8, 0) -- (8, 0);
	\draw[thick] (0, -8) -- (0, 8);
    \node[right] at (8,0) {$a$};
    \node[above] at (0,8) {$b\sigma$};
    \fill (9.5,5.7) rectangle (10.1,6.3); \node[right] at (10,6) {$= A(\GG)$};
    \fill (9.8,4.8) circle (0.2); \node[right] at (10,4.8) {$= \Z$};
    \draw (9.8,3.6) circle (0.2); \node[right] at (10,3.6) {$= \Z/2$};

    \fill (-0.3, -0.3) rectangle (0.3, 0.3);
    \fill (0, -7) circle(0.2);
    \fill (0, -6) circle(0.2);
    \fill (0, -5) circle(0.2);
    \fill (0, -4) circle(0.2);
    \fill (0, -3) circle(0.2);
    \fill (0, -2) circle(0.2);
    \fill (0, -1) circle(0.2);
    \fill (0, 1) circle(0.2);
    \fill (0, 2) circle(0.2);
    \fill (0, 3) circle(0.2);
    \fill (0, 4) circle(0.2);
    \fill (0, 5) circle(0.2);
    \fill (0, 6) circle(0.2);
    \fill (0, 7) circle(0.2);

    \fill (-2, 2) circle(0.2);
    \fill (-4, 4) circle(0.2);
    \fill (-6, 6) circle(0.2);

    \draw[fill=white] (-2, 3) circle(0.2);
    \draw[fill=white] (-2, 4) circle(0.2);
    \draw[fill=white] (-2, 5) circle(0.2);
    \draw[fill=white] (-2, 6) circle(0.2);
    \draw[fill=white] (-2, 7) circle(0.2);
    \draw[fill=white] (-4, 5) circle(0.2);
    \draw[fill=white] (-4, 6) circle(0.2);
    \draw[fill=white] (-4, 7) circle(0.2);
    \draw[fill=white] (-6, 7) circle(0.2);

    \fill (2, -2) circle (0.2);
    \fill (4, -4) circle (0.2);
    \fill (6, -6) circle (0.2);
    \draw[fill=white] (3, -3) circle(0.2);
    \draw[fill=white] (5, -5) circle(0.2);
    \draw[fill=white] (7, -7) circle(0.2);
    
    \draw[fill=white] (3, -4) circle(0.2);
    \draw[fill=white] (3, -5) circle(0.2);
    \draw[fill=white] (3, -6) circle(0.2);
    \draw[fill=white] (3, -7) circle(0.2);
    \draw[fill=white] (5, -6) circle(0.2);
    \draw[fill=white] (5, -7) circle(0.2);

    \node[right] at (0,1) {$e$};
    \node[below left] at (-2,2) {$\xi$};
    \node[above right] at (2,-2) {$\tau(\iota^{-2})$};
    \node[left] at (0,-1) {$e^{-1}\kappa$};

\end{tikzpicture}
\caption{$\HS^{RO(\GG)}$}\label{fig:cohompt}
\end{figure}

Recall that $A(\GG)$ is the Grothendieck group of finite $\GG$-sets, with multiplication given by products of sets.
Additively, it is free abelian on the classes of the orbits of $\GG$, for which we write
$1 = [\GG/\GG]$ and $g = [\GG/e]$. The multiplication is given by $g^2 = 2g$.
We also write $\kappa = 2 - g$. Other important elements are shown in the figure:
The group in degree $\sigma$ is generated by an element $e$,
while the group in degree $-2 + 2\sigma$ is generated by an element $\xi$.
The groups in the second quadrant are generated by the products $e^m\xi^n$, with $2e\xi = 0$.
We have $g\xi = 2\xi$ and $ge = 0$.
The groups in gradings $-m\sigma$, $m\geq 1$, are generated by elements $e^{-m}\kappa$, so named
because $e^m\cdot e^{-m}\kappa = \kappa$. We also have $ge^{-m}\kappa = 0$.

To explain $\tau(\iota^{-2})$, we think for moment about the nonequivariant cohomology
of a point. If we grade it on $RO(\GG)$, we get
$H^{RO(\GG)}(S^0;\Z) \iso \Z[\iota^{\pm 1}]$, where $\deg \iota = -1 + \sigma$.
(Nonequivariantly, we cannot tell the difference between $\R$ and $\R^\sigma$.)
We have $\rho(\xi) = \iota^2$ and $\tau(\iota^2) = g\xi = 2\xi$.
Note also that $\tau(1) = g$.
In the fourth quadrant we have that the group in grading $n(1-\sigma)$, $n\geq 2$, is
generated by $\tau(\iota^{-n})$.
The remaining groups in the fourth quadrant will not concern us here.
For more details, see \cite{Co:InfinitePublished} or \cite{CHTFiniteProjSpace}.

\subsection{The cohomology of projective spaces}\label{app:cohomologyprojective}
We summarize the cohomology of $\Xpq pq$ as calculated in \cite{CHTFiniteProjSpace}.

Write $BU(1) = \Xpq \infty\infty$.
Its fixed set is
\[
    BU(1)^\GG = \Xp \infty \disjunion \Xq \infty = B^0 \disjunion B^1,
\]
where we use the indices 0 and 1 to evoke the trivial and nontrivial representation
of $\GG$, respectively.
(We use this convention throughout, that an index 0 refers to something related to
$B^0$ and an index 1 refers to something related to $B^1$.)
Because $BU(1)$, $B^0$, and $B^1$ are all nonequivariantly simply connected,
representations of $\Pi BU(1)$ are determined by their restrictions to $B^0$ and $B^1$,
which are elements of $RO(\GG)$ that must have the same nonequivariant rank and
the same parity for the ranks of their fixed point representations.
That is, we can think of an element $\alpha\in RO(\Pi BU(1))$ as a pair
\[
    \alpha = (\alpha_0,\alpha_1) \in RO(\GG)\dirsum RO(\GG),
\]
where the ranks $|\alpha_0| = |\alpha_1|$
while $|\alpha_0^\GG| \equiv |\alpha_1^\GG| \pmod 2$.
For example, we have the representations
\begin{align*}
    \Omega_0 &= (2\sigma - 2, 0) \\
\intertext{and}
    \Omega_1 &= (0, 2\sigma - 2).
\end{align*}
As shown in \cite[\S2.2, p.13]{CHTFiniteProjSpace}, 
$RO(\Pi BU(1))$ is free abelian on three generators, but we find
the following presentation most useful:
\[
    RO(\Pi BU(1)) = \Z\{1,\sigma,\Omega_0,\Omega_1\}/\langle \Omega_0 + \Omega_1 = 2\sigma - 2\rangle.
\]
For brevity of notation we shall write
\[
    \HH^\gr(X) = \HH^{RO(\Pi BU(1))}(X)
\]
for any $\GG$-space over $BU(1)$.

Let $\omega$ denote the tautological line bundle over $BU(1)$, let $\omega\dual$ be its dual bundle, let
$\chi\omega = \omega\tensor_\C \C^\sigma$, and let $\chi\omega\dual$ be the dual of $\chi\omega$.
We will also use the notation from algebraic geometry in which 
$O(n) = (\omega\dual)^{\tensor n}$, so
$\omega = O(-1)$, $\omega\dual = O(1)$, $\chi O(-1) = \chi\omega$, and $\chi O(1) = \chi\omega\dual$.

Associated to any vector bundle over $BU(1)$ is a representation in $RO(\Pi BU(1))$ that we think of as the
equivariant rank or dimension of the bundle; this representation is given by the fiber representations
over $B^0$ and $B^1$. In the case of $\omega$ and $\chi\omega$, we have
\begin{align*}
    \omega &= 2 + \Omega_1 \\
    \chi\omega &= 2 + \Omega_0,
\end{align*}
where we write $\omega$ and $\chi\omega$ again for the associated elements of $RO(\Pi BU(1))$.

One of the main reasons for extending the grading of ordinary cohomology to $RO(\Pi X)$ is that
it then has a Thom isomorphism, so a vector bundle has an Euler class, whose grading is the dimension of the bundle.
For example, $\omega\dual$ and $\chi\omega\dual$ have Euler classes
\begin{align*}
    \cwd &\in \HH^{\omega}(BU(1)) = \HH^{2 + \Omega_1}(BU(1)) \\
\intertext{and}
    \cxwd &\in \HH^{\chi\omega}(BU(1)) = \HH^{2 + \Omega_0}(BU(1)).
\end{align*}

The cohomology of $BU(1)$ was calculated in \cite{Co:InfinitePublished} as follows.

\begin{theorem}[{\cite[Theorem 11.3]{Co:InfinitePublished}}]
$H_{\GG}^\gr(BU(1))$ is an algebra over $\HS$ generated by the 
Euler classes $\cwd$ and $\cxwd$ together with classes $\zeta_0$ and $\zeta_1$. 
These elements live in gradings
\begin{alignat*}{4}
 \grad\cwd &= \omega  \qquad&  \grad\cxwd &= \chiw  \\
 \grad\zeta_1 &= \omega - 2  \qquad& \grad\zeta_0 &= \chiw-2 
\end{alignat*}
They satisfy the relations
\begin{align*}
		\zeta_0 \zeta_1 &= \xi \\
        \mathllap{\text{and}\quad}\zeta_1 \cxwd &= (1-\kappa)\zeta_0 \cwd + e^2,
\end{align*}
and these relations completely determine the algebra.
Moreover, $H_{\GG}^\gr(BU(1))$ is free as a module over $\HS$.
\qed
\end{theorem}

The freeness stated in the last part of the theorem is a phenomenon
that occurs in a large class of examples, as shown
by Ferland and Lewis in \cite{FerlandLewis}.
See also \cite{KronholmFree}, \cite{HogleMayFreeness}, and \cite{DuggerHazelMayModules}.

There are two restriction maps we will use,
\begin{align*}
    \rho\colon H_\GG^{\alpha}(BU(1)) &\to H^{\alpha}(BU(1)),
\intertext{restriction to nonequivariant cohomology, and}
    (-)^\GG\colon H_\GG^\alpha(BU(1)) &\to H^{\alpha_0^\GG}(B^0) \dirsum H^{\alpha_1^\GG}(B^1),
\end{align*}
the fixed-point map. 
In the case of $\rho$, we are considering nonequivariant cohomology to
be graded on $RO(\Pi BU(1))$ via the forgetful or rank homomorphism
\[
    RO(\Pi_\GG BU(1))\to RO(\Pi_e BU(1)) \iso \ZZ,
\]
which is just the map $\alpha\mapsto |\alpha|$.
When we impose this regrading, $H^\gr(BU(1))$ acquires two invertible elements,
\[
    \iota \in H^{\sigma-1}(BU(1)) \qquad\text{and}\qquad \zeta\in H^{\omega-2}(BU(1))
\]
expressing the fact that $\sigma - 1$ and $\omega-2$ generate the kernel
of the rank homomorphism.

The restriction maps are ring maps and their values on the multiplicative generators are
\begin{equation}\label{eqn:restrictions}
\begin{aligned}
    \rho(\zeta_0) &= \iota^2\zeta^{-1} &  \zeta_0^\GG &= (0,1) \\
    \rho(\zeta_1) &= \zeta & \zeta_1^\GG &= (1,0) \\
    \rho(\cwd) &= \zeta \cd & \cwd[\GG] &= ( \cd, 1) \\
    \rho(\cxwd) &= \iota^2\zeta^{-1} \cd \qquad\qquad& \cxwd[\GG] &= (1,  \cd).
\end{aligned}
\end{equation}
Here, $ \cd \in H^2(BU(1))$ denotes the first nonequivariant Chern class of $\omega\dual = O(1)$.
There are similar restriction maps
\begin{align*}
    \rho\colon \HS^\alpha &\to H^{\alpha}(S^0), \\
    (-)^\GG\colon \HS^\alpha &\to H^{\alpha^\GG}(S^0).
\end{align*}
The values on the most interesting elements are
\begin{equation}\label{eqn:restrpoint}
\begin{aligned}
    \rho(\xi^k) &= \iota^{2k} & \rho(\tau(\iota^{2k})) &= 2\iota^{2k} & \rho(e^{-k}\kappa) &= 0 & \rho(e^k) &= 0 \\
    (\xi^k)^\GG &= 0 & \tau(\iota^{2k})^\GG &= 0 & (e^{-k}\kappa)^\GG &= 2 & (e^k)^\GG &= 1.
\end{aligned}
\end{equation}

We also have the Frobenius relation
\begin{equation}\label{eqn:Frobenius}
    b\tau(a)=\tau(\rho(b)a)
\end{equation}
for $b\in \HH^\gr(BU(1))$ and $a\in H^\gr(BU(1))$,
which is very useful for calculations involving $\tau$.

Moving now to the finite projective space $\Xpq pq$,
we think of it as a space over $BU(1)$ by the evident inclusion
$\Xpq pq \includesin \Xpq\infty\infty$.
This inclusion induces an isomorphism $RO(\Pi BU(1)) \iso RO(\Pi\Xpq pq)$
when $p, q > 0$ and an epimorphism $RO(\Pi BU(1)) \onto RO(\Pi\Xpq pq)$ if either is zero,
so we will grade the cohomology of $\Xpq pq$ on $RO(\Pi BU(1))$ for all $p$ and $q$.
The elements $\cwd$, $\cxwd$, $\zeta_0$, and $\zeta_1$ pull back to elements in
$\HH^\gr(\Xpq pq)$ we will call by the same names.
In \cite{CHTFiniteProjSpace}, we showed the following, though stated somewhat differently.



\begin{theorem}[{\cite[Theorem A]{CHTFiniteProjSpace}}]\label{thm:cohomStructure}
Let $0 \leq p < \infty$ and $0 \leq q < \infty$ with $p+q > 0$.
Then $H_{\GG}^\gr(\Xpq{p}{q})$ is a free module over $\HS$.
As a (graded) commutative algebra over $\HH^\gr(BU(1))$,
it is generated by elements $\zeta_0^{-k}\cwd[p]$ and $\zeta_1^{-k}\cxwd[q]$
for $k\geq 1$, subject to the relations
\begin{align*}
    \zeta_0\cdot \zeta_0^{-k}\cwd[p] &= \zeta_0^{-(k-1)}\cwd[p] && k\geq 1 \\
    \zeta_1\cdot \zeta_1^{-k}\cxwd[q] &= \zeta_1^{-(k-1)}\cxwd[q] && k\geq 1 \\
    \cwd[p]\cxwd[q] &= 0.
\end{align*}
\qed
\end{theorem}

As shorthand for the existence of the elements $\zeta_0^{-k}\cwd[p]$, we often
say that ``$\cwd[p]$ is infinitely divisible by $\zeta_0$,''
and similarly for the elements $\zeta_1^{-k}\cxwd[q]$.

\subsection{Pushforwards and representation by singular manifolds}
Because ordinary cohomology has Thom isomorphisms and Poincar\'e duality,
we can define some useful pushforward maps.
We state the following definition for $\GG$-manifolds, but it works
equally well for any compact Lie group $G$.

\begin{definition}
Let $f\colon N\to M$ be a map of closed smooth $\GG$-manifolds with
$\dim N \in RO(\Pi M)$, that is, $\dim N = f^*\gamma$ for some $\gamma\in RO(\Pi M)$;
for our purposes here we choose one such $\gamma$ and call it $\dim N$.
Let $\nu = \dim M - \dim N\in RO(\Pi M)$.
Then the pushforward map
\[
    f_!\colon \HH^\alpha(N) \to \HH^{\alpha + \nu}(M)
\]
is defined for $\alpha\in RO(\Pi M)$ 
by the following diagram, in which the vertical maps are Poincar\'e duality isomorphisms:
\[
    \xymatrix{
        \HH^{\alpha}(N) \ar@{-->}[r]^-{f_!} \ar[d]_\iso & \HH^{\alpha+\dim M - \dim N}(M) \ar[d]^\iso \\
        H^\GG_{\dim N - \alpha}(N) \ar[r]_-{f_*} & H^\GG_{\dim N - \alpha}(M).
    }
\]
\end{definition}

This pushforward map has the usual properties familiar from the nonequivariant case, including
the fact that it is a map of $\HH^{RO(\Pi M)}(M)$-modules, meaning that
\[
    x f_!(y) = f_!(f^*(x)y) \text{ for } x\in \HH^{RO(\Pi M)}(M) \text{ and } y\in \HH^{RO(\Pi M)}(N).
\]

We will use pushforward maps as above in our calculation of the cohomology of quadrics.
They also give a useful way of thinking about cohomology classes geometrically.

\begin{definition}
Let $f\colon N\to M$ be a map of closed smooth $\GG$-manifolds with
$\dim N \in RO(\Pi M)$ and let $\nu = \dim M - \dim N\in RO(\Pi M)$. We define
\[
    [N]^* = f_!(1) \in \HH^{\nu}(M).
\]
\end{definition}

Unwinding the definitions, $[N]^*$ is the cohomology class Poincar\'e dual to $f_*[N] \in H^\GG_{\dim N}(M)$.
In general, not all cohomology classes can be represented in this way, but a refinement of this idea
does work in general. We summarize a more detailed discussion from \cite{CH:geometric},
which we simplify here to discuss only singular manifolds without boundary.

\begin{definition}
A \emph{$\gamma$-dimensional closed singular $\GG$-manifold} of depth $n$ is a space of the form
$X = N \union_{\bndry N} M\phi$ where $(N,\bndry N)$ is a smooth compact $\gamma$-dimensional
$\GG$-manifold with boundary,
$P$ is a $(\gamma-k)$-dimensional closed singular $\GG$-manifold of depth $n-1$ for some $k\geq 0$,
and $M\phi$ is the mapping cylinder of a map $\phi\colon \bndry N \to P$.
The base of this recursive definition is that a singular manifold of depth 0 is
a smooth manifold.
We will use the notation $(X;P)$ for a singular manifold $X$ with singular part $P$
when we need to be so explicit.
\end{definition}

\begin{definition}
A singular $\GG$-manifold \emph{with codimension-2 singularities} is one of dimension $\gamma$ in which the singular
part $P$ has dimension $\gamma-k$ for some $k\geq 2$.
\end{definition}

The point of this last requirement is that, if $M = N\union_{\bndry N} M\phi$ is a
$\gamma$-dimensional singular $\GG$-manifold with codimension-2 singularities, then
we have an isomorphism
\[
    H^\GG_\gamma(M) \iso H^\GG_\gamma(N,\bndry N).
\]
This allows us to define a fundamental class $[M]\in H^\GG_\gamma(M)$ to be the element corresponding to
the fundamental class $[N,\bndry N]\in H^\GG_\gamma(N,\bndry N)$ that exists
because $(N,\bndry N)$ is a smooth $\GG$-manifold.
If $f\colon M\to Q$ is a map from $M$ to a smooth closed $\GG$-manifold $Q$ and
$\dim M\in RO(\Pi Q)$, we can define the pushforward again by the following diagram.
\[
    \xymatrix{
        \HH^{\alpha}(M) \ar@{-->}[r]^-{f_!} \ar[d]_{\cap[M]}
            & \HH^{\alpha+\dim Q - \dim M}(Q) \ar[d]^\iso \\
        H^\GG_{\dim M - \alpha}(M) \ar[r]_-{f_*} & H^\GG_{\dim M - \alpha}(Q).
    }
\]
The left vertical map need not be an isomorphism, but we need only the right vertical
map to be one in order to define $f_!$.
Our primary use of this general pushforward map will be the following.

\begin{definition}
Let $f\colon N\to M$ be a map where
$(N;P)$ is a singular $\GG$-manifold with codimension-2 singularities $P$,
$M$ is a smooth closed $\GG$-manifold, and $\dim N\in RO(\Pi M)$. Then
we define
\[
    [N;P]^* = f_!(1) \in \HH^{\dim M - \dim N}(M).
\]
\end{definition}

Again, this means that $[N;P]^*$ is the cohomology class Poincar\'e dual
to $f_*[N]$.
A result of Hastings and Waner in \cite{HM:singularities} shows that, after passing further
to ``virtual manifolds,'' every homology element
can be represented in this way and that, in fact, the ordinary homology of $M$ is isomorphic
to the cobordism group of virtual singular manifolds over $M$.
We will not use the full force of this result, but will represent elements using singular manifolds
when it makes our arguments and results clearer.

\begin{example}
The element $\zeta_0\in \HH^{\Omega_0}(\Xpq pq)$ can be represented as
\[
    \zeta_0 = [\Xpq pq; \Xpq{p}{(q-1)}]^*.
\]
Here, we let $(N,\bndry N)\subset \Xpq pq$ be the complement of an open normal tube around
$\Xpq p{(q-1)}$, with $\phi\colon \bndry N\to \Xpq p{(q-1)}$ the projection of the sphere bundle
of the normal bundle.
Because $N^\GG\intersect \Xp p = \emptyset$, we may reinterpret the dimension of $N$ to be
\[
    \dim N = \dim\Xpq pq - \Omega_0.
\]
We have
\[
    \dim \Xpq p{(q-1)} = \dim\Xpq pq - \chi\omega = \dim\Xpq pq - \Omega_0 - 2,
\]
so $(\Xpq pq; \Xpq p{(q-1)})$ is a $(\dim \Xpq pq - \Omega_0)$-dimensional
singular manifold with codimension-2 singularities. With a little more work
we can show that $[\Xpq pq;\Xpq p{(q-1)}]^* = \zeta_0$ as claimed.
\end{example}

\bibliography{Bibliography}{}
\bibliographystyle{amsplain} 

\end{document}